\documentclass[11pt]{amsart}
\usepackage[margin=1.15in]{geometry}
\usepackage{amsmath,amssymb,amsfonts,amsthm}
\usepackage{mathrsfs}
\usepackage{graphicx}
\usepackage{float}
\usepackage{booktabs}
\usepackage[colorlinks=true,linkcolor=blue,citecolor=blue,urlcolor=black]{hyperref}
\numberwithin{equation}{section}

\theoremstyle{plain}
\newtheorem{theorem}{Theorem}[section]
\newtheorem{proposition}[theorem]{Proposition}
\newtheorem{lemma}[theorem]{Lemma}
\newtheorem{corollary}[theorem]{Corollary}
\theoremstyle{definition}
\newtheorem{definition}[theorem]{Definition}
\newtheorem{remark}[theorem]{Remark}

\newcommand{\R}{\mathbb{R}}
\newcommand{\E}{\mathbb{E}}
\newcommand{\Prob}{\mathbb{P}}
\newcommand{\dist}{\operatorname{dist}}
\newcommand{\Leb}{\operatorname{Leb}}
\newcommand{\Dgm}{\operatorname{Dgm}}
\newcommand{\card}{\operatorname{card}}
\newcommand{\diam}{\operatorname{diam}}
\newcommand{\eqd}{\overset{d}{=}}
\newcommand{\norm}[1]{\lVert #1 \rVert}
\newcommand{\cL}{\mathcal{L}}
\newcommand{\shape}{\mathfrak{s}}
\newcommand{\cin}{c_{\mathrm{in}}}
\newcommand{\Nin}{N_{\mathrm{in}}}
\newcommand{\Ncrit}{\mathcal{N}^{\mathrm{crit}}}
\newcommand{\Nres}{\mathcal{N}^{\mathrm{crit}}_{\circ}}

\begin{document}

\title[The inradius law for the cavities of planar Brownian motion]{The inradius law for the cavities of planar Brownian motion, with applications to the Wiener sausage}

\author{Tristan Guillaume}
\address{Laboratoire Thema, CY Cergy Paris Universit\'e, 33 boulevard du Port, F-95011 Cergy-Pontoise Cedex, France}
\email{tristan.guillaume@cyu.fr}


\begin{abstract}Let $K_{1}=B[0,1]$ be the trace of a planar Brownian motion, and let $\Nin(\eta)$ be the number of its cavities, the bounded components of $\R^{2}\setminus K_{1}$, whose inradius exceeds $\eta$. We prove that, almost surely and in $L^{p}$ for every $p\in[1,2]$,
\[
\eta^{2}\bigl(\log(1/\eta)\bigr)^{2}\,\Nin(\eta)\longrightarrow\cin=\frac{\pi}{2}\int G^{2}\,d\cL\ \in\bigl(0,\tfrac12\bigr],
\]
where $G$ is the ratio of the inradius of a shape to the square root of its area and $\cL$ is Werner's limiting shape law. The inradius, and not the area by which the cavities have been counted classically, decides survival under thickening: a cavity contains a hole of the Wiener sausage of radius $r$ if and only if its inradius exceeds $r$. Through Brownian scaling, the number $\beta_{1}^{T}(r)$ of holes of the sausage of radius $r$ of $B[0,T]$, the Betti curve of its degree-one persistent homology, is bounded below at the scale $T/(r^{2}(\log T)^{2})$, with the identified constant $4\cin$ in expectation and with any smaller constant with probability tending to one. A matching upper bound cannot follow from any general bound in terms of perimeter, area, diameter or length; we prove that it is equivalent to a first-moment estimate for near-critical pockets, the holes at level $\eta$ on which the distance to the trace is at most $2\eta$. That estimate is proved in a companion paper, so that the expected number of holes is of exact order $T/(r^{2}(\log T)^{2})$.
\end{abstract}

\maketitle

\section{Introduction}\label{sec:intro}

\subsection{Holes of the planar Wiener sausage}

Let $B=(B_{t})_{t\ge0}$ be a standard planar Brownian motion, let $K_{T}:=B[0,T]$ be its trace up to time $T$, and, for $r>0$, let
\[
K_{T}^{(r)}:=\{x\in\R^{2}:\dist(x,K_{T})\le r\}
\]
be the Wiener sausage of radius $r$. The large-time behaviour of the sausage and the fine structure of the complement of the Brownian path are classical objects of planar stochastic geometry \cite{Sp64,DV75,LG90,LG91,LG92}. In dimension two, recurrence produces logarithmic effects: for instance, the mean area of $K_{T}^{(r)}$ is of order $T/\log T$.

This paper studies a topological observable of the sausage, its number of holes. Since $K_{T}$ is connected, so is $K_{T}^{(r)}$, and its first Betti number is
\[
\beta_{1}^{T}(r):=\beta_{1}\bigl(K_{T}^{(r)}\bigr)=\card\bigl\{\text{bounded connected components of }\R^{2}\setminus K_{T}^{(r)}\bigr\}.
\]
Considering all radii at once, $(K_{T}^{(r)})_{r\ge0}$ is the Wiener-sausage filtration; its degree-one persistence diagram $\Dgm_{1}(K_{T})$ records the radii at which holes are born and filled, and $r\mapsto\beta_{1}^{T}(r)$ is its Betti curve: at regular radii, the hole count is the number of persistence intervals alive at radius $r$ \cite{ELZ02,CEH07,CDGO16}, and the two agree for almost every radius and, for the Brownian trace, almost surely at each fixed radius (Section~\ref{subsec:persistence}). For a bounded Borel weight $\psi$ supported in a compact interval $[r_{0},r_{1}]\subset(0,\infty)$ we consider the smoothed Betti-curve observable
\[
\Phi_{\psi}(T):=\int_{r_{0}}^{r_{1}}\beta_{1}^{T}(r)\,\psi(r)\,dr ,
\]
which is a linear statistic of the persistence diagram (Section~\ref{subsec:persistence}). The question is the large-time order of $\beta_{1}^{T}(r)$ and $\Phi_{\psi}(T)$, and the constant in front of it.

For Brownian motion with drift, a regeneration scheme yields laws linear in $T$ for such observables \cite{GuDrift}. Without drift the path is recurrent and this structure disappears; it is replaced by self-similarity. Since $K_{T}\eqd\sqrt{T}K_{1}$, the whole filtration at time $T$ is equal in law to the unit-time filtration dilated by $\sqrt{T}$, and
\[
\Dgm_{1}(K_{T})\eqd\sqrt{T}\cdot\Dgm_{1}(K_{1}),\qquad\beta_{1}^{T}(r)\eqd\beta_{1}\bigl(K_{1}^{(r/\sqrt{T})}\bigr).
\]
A large-time problem at a fixed radius is thus a small-radius problem for the unit-time trace $K_{1}$. At small radii, many holes of $K_{1}^{(\varepsilon)}$ come from the bounded components of $\R^{2}\setminus K_{1}$, the \emph{cavities} of the trace (Figure~\ref{fig:sim}), which survive the thickening by $\varepsilon$ if and only if their inradius exceeds $\varepsilon$. Area alone does not decide survival: a long thin cavity of large area has a small inradius and disappears as soon as the radius exceeds it. This makes the inradius, rather than the area by which the cavities have been counted classically \cite{Mo89,LG91,LG92}, the relevant size.

\begin{figure}[t]
\centering
\includegraphics[width=\textwidth]{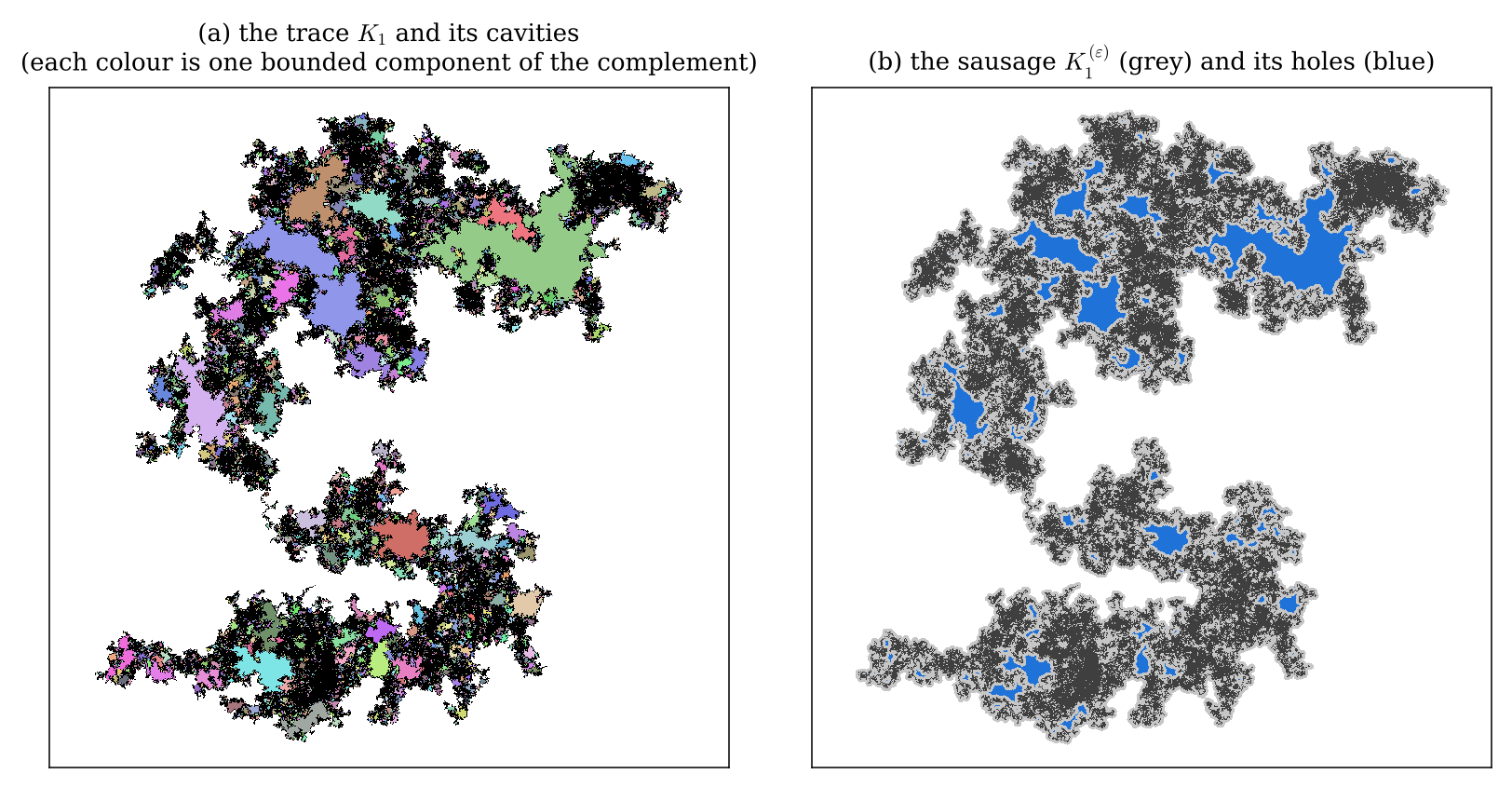}
\caption{A simulated planar Brownian path at unit time, on a $1400\times1400$ grid. (a) The trace $K_{1}$ (black) and its cavities, each in its own colour; at this resolution about $18\,000$ cavities are visible. (b) The sausage $K_{1}^{(\varepsilon)}$ for a small $\varepsilon$ (grey) and its holes (blue), $197$ here. Almost all cavities are too small or too thin to survive the thickening. The counts depend on the resolution and on $\varepsilon$; the figure is only an illustration.}\label{fig:sim}
\end{figure}

\subsection{Main results}\label{subsec:mainresults}

Write $\ell_{\eta}:=\log(1/\eta)$ for $\eta\in(0,1)$. Let $(A_{i})_{i\ge1}$ be the cavities of $K_{1}$, ranked by nonincreasing area $L_{i}:=\Leb_{2}(A_{i})$, let $\rho_{i}:=\sup_{x\in A_{i}}\dist(x,K_{1})$ be the inradius of $A_{i}$, and set
\[
M(u):=\card\{i:L_{i}>u\},\qquad\Nin(\eta):=\card\{i:\rho_{i}>\eta\}.
\]
Le Gall proved that $u\ell_{u}^{2}M(u)\to2\pi$ almost surely \cite{LG91,LG92}, and Werner proved that the shapes of the cavities, taken in the order of decreasing area, have a limiting empirical distribution $\cL$ \cite{We94}. For a nonempty bounded open set $A$ let $G(A):=\rho_{\partial A}(A)/\sqrt{\Leb_{2}(A)}\in(0,\pi^{-1/2}]$, where $\rho_{\partial A}(A)=\sup_{x\in A}\dist(x,\partial A)$; $G$ is a function of the shape, and $\rho_{i}=G(A_{i})\sqrt{L_{i}}$.

\begin{theorem}[Cavity inradius law]\label{thm:A}
Almost surely,
\[
\lim_{\eta\downarrow0}\eta^{2}\ell_{\eta}^{2}\,\Nin(\eta)=\cin:=\frac{\pi}{2}\int G^{2}\,d\cL ,\qquad 0<\cin\le\frac12 .
\]
Moreover $\E[M(u)^{2}]\sim(2\pi)^{2}u^{-2}\ell_{u}^{-4}$ as $u\downarrow0$, and both $u\ell_{u}^{2}M(u)\to2\pi$ and $\eta^{2}\ell_{\eta}^{2}\Nin(\eta)\to\cin$ hold in $L^{p}$ for every $p\in[1,2]$; in particular $\E[\Nin(\eta)]\sim\cin\,\eta^{-2}\ell_{\eta}^{-2}$.
\end{theorem}

Every cavity of inradius larger than $\varepsilon$ contains a hole of $K_{1}^{(\varepsilon)}$, and distinct cavities give distinct holes; Brownian scaling converts $\varepsilon=r/\sqrt{T}$ into the large-time normalization, with $\varepsilon^{-2}\ell_{\varepsilon}^{-2}\sim4T/(r^{2}(\log T)^{2})$. For a weight $\psi$ as above put $I_{\psi}:=\int_{r_{0}}^{r_{1}}\psi(r)r^{-2}\,dr$.

\begin{theorem}[Lower bounds with an identified constant]\label{thm:B}
Let $0<r_{0}<r_{1}<\infty$ and let $\psi\ge0$ be bounded and Borel, supported in $[r_{0},r_{1}]$.
\begin{itemize}
\item[(a)] Almost surely, simultaneously for all small radii,
\[
\liminf_{\varepsilon\downarrow0}\varepsilon^{2}\ell_{\varepsilon}^{2}\,\beta_{1}\bigl(K_{1}^{(\varepsilon)}\bigr)\ge\cin,\qquad
\liminf_{\varepsilon\downarrow0}\varepsilon^{2}\ell_{\varepsilon}^{2}\int_{r_{0}}^{r_{1}}\beta_{1}\bigl(K_{1}^{(\varepsilon r)}\bigr)\psi(r)\,dr\ge\cin I_{\psi}.
\]
\item[(b)] For every fixed $r>0$ and every $\theta\in(0,1)$,
\[
\begin{gathered}
\liminf_{T\to\infty}\frac{(\log T)^{2}}{T}\,\E\bigl[\beta_{1}^{T}(r)\bigr]\ge\frac{4\cin}{r^{2}},\\
\lim_{T\to\infty}\Prob\Bigl(\beta_{1}^{T}(r)\ge(1-\theta)\frac{4\cin}{r^{2}}\,\frac{T}{(\log T)^{2}}\Bigr)=1 .
\end{gathered}
\]
\item[(c)] For every $\theta\in(0,1)$,
\[
\begin{gathered}
\liminf_{T\to\infty}\frac{(\log T)^{2}}{T}\,\E\bigl[\Phi_{\psi}(T)\bigr]\ge4\cin I_{\psi},\\
\lim_{T\to\infty}\Prob\Bigl(\Phi_{\psi}(T)\ge(1-\theta)\,4\cin I_{\psi}\,\frac{T}{(\log T)^{2}}\Bigr)=1 .
\end{gathered}
\]
\end{itemize}
These bounds are attained by the cavities: the number $N_{T}(r)$ of bounded components of $\R^{2}\setminus K_{T}$ of inradius larger than $r$ satisfies $(\log T)^{2}T^{-1}N_{T}(r)\to4\cin/r^{2}$ in probability and in $L^{p}$, $p\in[1,2]$.
\end{theorem}

Brownian scaling also forces the radius profile: if $(\log T)^{2}T^{-1}\E[\beta_{1}^{T}(r)]\to f(r)$ for every $r$, then $f(r)=\kappa/r^{2}$ for a constant $\kappa$ (Proposition~\ref{prop:profile}), and Theorem~\ref{thm:B}(b) gives $\kappa\ge4\cin$. The scale $T/(\log T)^{2}$ is one logarithmic factor below the area of the sausage, and it is the exact order of $\E[\beta_{1}^{T}(r)]$ (Corollary~\ref{cor:order} below).

The upper bound is of a different nature. No general bound on the number of holes of a planar parallel set at a prescribed radius can be given in terms of its perimeter, area, diameter or length \cite{GuPS}: almost tangent pieces of the set enclose holes of almost vanishing size. Deterministic estimates using finer information remain possible, and Lemma~\ref{lem:recursion} below is one, but the order of magnitude requires probabilistic input: such near-tangencies must be shown to be rare for the Brownian trace. Call a bounded component $U$ of $\{x:\dist(x,K_{1})>\eta\}$, that is, a hole of $K_{1}^{(\eta)}$, a \emph{near-critical pocket} at level $\eta$ if $\max_{\overline U}\dist(\cdot,K_{1})\le2\eta$, and let $\Ncrit(\eta)$ be their number. A cavity with inradius in $(\eta,2\eta]$ contains at least one near-critical pocket; let $\Nres(\eta):=\Ncrit(\eta)-\card\{i:\rho_{i}\in(\eta,2\eta]\}\ge0$ be the number of \emph{residual} near-critical pockets. Beyond one pocket for each cavity of inradius in $(\eta,2\eta]$, they are the other near-critical pockets lying in cavities, which arise when a cavity fragments under the thickening, and the near-critical pockets lying in the unbounded component of $\R^{2}\setminus K_{1}$; which pocket of such a cavity is regarded as its compulsory one does not affect the count.

\begin{theorem}[Near-critical reduction]\label{thm:C}
The following are equivalent.
\begin{itemize}
\item[(i)] There are $C_{+}<\infty$ and $\varepsilon_{0}>0$ with $\E[\beta_{1}(K_{1}^{(\varepsilon)})]\le C_{+}\varepsilon^{-2}\ell_{\varepsilon}^{-2}$ for $0<\varepsilon\le\varepsilon_{0}$.
\item[(ii)] \emph{Hypothesis (NC):} there are $C<\infty$ and $\eta_{0}>0$ with $\E[\Ncrit(\eta)]\le C\eta^{-2}\ell_{\eta}^{-2}$ for $0<\eta\le\eta_{0}$.
\item[(iii)] \emph{Hypothesis (NC$_{\circ}$):} there are $C_{\circ}<\infty$ and $\eta_{0}>0$ with $\E[\Nres(\eta)]\le C_{\circ}\eta^{-2}\ell_{\eta}^{-2}$ for $0<\eta\le\eta_{0}$.
\end{itemize}
If (iii) holds with the constant $C_{\circ}$, then $\limsup_{\varepsilon\downarrow0}\varepsilon^{2}\ell_{\varepsilon}^{2}\,\E[\beta_{1}(K_{1}^{(\varepsilon)})]\le\cin+\frac43C_{\circ}$, and, with $c_{+}:=4(\cin+\frac43C_{\circ})$, for every $r>0$ and every $\psi$ as in Theorem~\ref{thm:B},
\[
\limsup_{T\to\infty}\frac{(\log T)^{2}}{T}\,\E\bigl[\beta_{1}^{T}(r)\bigr]\le\frac{c_{+}}{r^{2}},\qquad
\limsup_{T\to\infty}\frac{(\log T)^{2}}{T}\,\E\bigl[\Phi_{\psi}(T)\bigr]\le c_{+}I_{\psi}.
\]
\end{theorem}

Hypothesis (NC) is proved in the companion paper \cite{GuII}, Theorem~1.2 there: each near-critical pocket is pinned by a constrained configuration of contacts of the trace at its own scale, and the expected number of such configurations is estimated by exterior harmonic measure, with one factor $1/\ell_{\eta}$ for each of the two independent arms of the path from a deterministic time. Since $\Nres\le\Ncrit$, Hypothesis (NC$_{\circ}$) holds as well, and Theorems~\ref{thm:B} and~\ref{thm:C} give the order of the hole count; the direct implication (ii)$\Rightarrow$(i) gives the better upper constant.

\begin{corollary}[The order of the hole count]\label{cor:order}
Let $C$ be a constant with which Hypothesis (NC) holds \cite[Theorem~1.2]{GuII}. For every $r>0$ and every $\psi$ as in Theorem~\ref{thm:B},
\[
\frac{4\cin}{r^{2}}\le\liminf_{T\to\infty}\frac{(\log T)^{2}}{T}\,\E\bigl[\beta_{1}^{T}(r)\bigr]\le\limsup_{T\to\infty}\frac{(\log T)^{2}}{T}\,\E\bigl[\beta_{1}^{T}(r)\bigr]\le\frac{16C}{3r^{2}},
\]
and $4\cin I_{\psi}\le\liminf\le\limsup\le\frac{16}{3}C\,I_{\psi}$ for $(\log T)^{2}T^{-1}\E[\Phi_{\psi}(T)]$ when $\psi\ge0$. In particular $\E[\beta_{1}^{T}(r)]\asymp T/(r^{2}(\log T)^{2})$ as $T\to\infty$, for each fixed $r>0$, and $(\log T)^{2}T^{-1}\beta_{1}^{T}(r)$ is bounded in probability as $T\to\infty$.
\end{corollary}

\begin{proof}
The lower bounds are Theorem~\ref{thm:B}(b)--(c). By \cite[Theorem~1.2]{GuII}, item~(ii) of Theorem~\ref{thm:C} holds with the constant $C$, and the implication (ii)$\Rightarrow$(i) in Section~\ref{subsec:proofC} gives
\[
\limsup_{\varepsilon\downarrow0}\varepsilon^{2}\ell_{\varepsilon}^{2}\,\E\bigl[\beta_{1}\bigl(K_{1}^{(\varepsilon)}\bigr)\bigr]\le\tfrac43C.
\] The fixed-radius and smoothed upper bounds follow from this exactly as in the last two steps of that proof, with $c_{+}$ replaced by $\frac{16}{3}C$. The last assertion follows from the upper bound by Markov's inequality. The proof of \cite[Theorem~1.2]{GuII} uses from the present paper only the measurability and duality facts of Section~\ref{sec:setup}, so there is no circularity.
\end{proof}

\begin{remark}[The lower constant]\label{rem:lowerconst}
The constant $4\cin$ of Corollary~\ref{cor:order} counts one hole per cavity of inradius larger than $r$. It is not sharp: by \cite[Theorem~1.4]{GuII}, the holes created when cavities fragment under the thickening contribute at leading order, and the lower bound improves to $4(\cin+c_{*})/r^{2}$ for an absolute constant $c_{*}>0$; this is the form displayed in \cite[Theorem~1.1]{GuII}. The decomposition behind this is Proposition~\ref{prop:threeterm}; see also Section~\ref{sec:outlook}.
\end{remark}

The equivalence of (i) and (ii) is deterministic, apart from the finiteness of the moments of $\diam K_{1}$: a dyadic depth recursion charges every hole at level $\varepsilon$ injectively to a near-critical pocket at one of the levels $2^{j}\varepsilon$, except for deep holes, which are counted by packing. The passage to (iii) uses Theorem~\ref{thm:A}: the part of (NC) carried by one pocket per cavity is a theorem, and what remains is the count of residual near-critical pockets, which is bounded in \cite{GuII}. The holes of $K_{1}^{(\varepsilon)}$ split exactly into one pocket per surviving cavity, the additional pockets of fragmenting cavities, and the bounded pockets of the unbounded component (Proposition~\ref{prop:threeterm} and Fig.~\ref{fig:cavities}(c)); the first family is governed by Theorem~\ref{thm:A}; the other two are at most of the same order, by Corollary~\ref{cor:order}. The third is in fact of lower order: its expectation is $O(\varepsilon^{-2+\theta})$ for an absolute $\theta>0$ \cite[Theorem~1.3]{GuII}, so that at leading order the holes come from the cavities. The second is of exact order $\varepsilon^{-2}\ell_{\varepsilon}^{-2}$ in expectation \cite[Theorem~1.4]{GuII} (Remark~\ref{rem:lowerconst}); the existence of the limit of $\varepsilon^{2}\ell_{\varepsilon}^{2}\E[N_{\mathrm{split}}(\varepsilon)]$ is open.

\subsection{What the results add}\label{subsec:new}

The probabilistic inputs are Le Gall's almost-sure law for $M(u)$, Werner's empirical shape theorem, and the window estimates of Le Gall and Werner on which these laws rest; we quote each in the exact form used (Sections~\ref{subsec:inputs} and~\ref{sec:moments}). Three points go beyond them.

\emph{From ranked areas and shapes to a moving inradius threshold.} By $\rho_{i}=G(A_{i})\sqrt{L_{i}}$, a cavity is counted in $\Nin(\eta)$ if and only if $G(A_{i})>\eta/\sqrt{L_{i}}$. The threshold depends on the rank and on the scale, and the relevant ranks, of order $a_{\eta}=\pi/(2\eta^{2}\ell_{\eta}^{2})$, tend to infinity with $\eta^{-1}$. Werner's theorem controls Ces\`aro averages over the $n$ largest cavities of a \emph{fixed} bounded test function, with a null set depending on the test function; it says nothing directly about a threshold that moves with the rank. We localize it to geometric blocks of ranks by differencing (Fig.~\ref{fig:mechanism}), invert Le Gall's law into a rank--area law, and bracket the moving threshold multiplicatively on each block. The grid errors are absorbed exactly by a change of variables in the layer-cake formula for $\int G^{2}d\cL$, so that no continuity of the law of $G$ is required, and only countably many shape sets are used. No independence between areas and shapes is assumed: the limit has the form that independence would suggest, but the proof needs only the empirical law along the area ranking. The same argument applies to any bounded measurable shape functional in place of $G$, almost surely and in $L^{p}$, $p\in[1,2]$ (Remark~\ref{rem:otherH}).

\emph{Moment control.} Almost-sure laws do not give expectations, and the expected number of holes is the natural quantity for persistence. The second-moment asymptotic of Theorem~\ref{thm:A} follows from the window estimates by a deterministic $L^{2}$ summation over a geometric tiling of the area axis, in which the leading coefficient of the window estimates is kept and the window ratio is then sent to one. It gives convergence in mean square for the inradius count and, through Brownian scaling, the expectation and probability statements of Theorem~\ref{thm:B}; in Theorem~\ref{thm:C} it is what makes the cavity part of Hypothesis (NC) unconditional.

\emph{Consequences for the topology of the sausage.} Existing component counts for the sausage do not give the total hole count at a fixed radius. Honzl \cite[Theorem~2]{Ho14} extends Le Gall's law to the components of $\R^{2}\setminus K_{1}^{(\varepsilon)}$ of area at least $u$, almost surely and uniformly for $0<\varepsilon<(u/\pi)^{b}$, for every fixed $b>1/2$: the radius of the sausage must be polynomially smaller than the linear scale $\sqrt{u/\pi}$ of the components counted. Taking $u=\pi(2\varepsilon)^{1/b}$, this yields $\liminf_{\varepsilon\downarrow0}\varepsilon^{1/b}\ell_{\varepsilon}^{2}\,\beta_{1}(K_{1}^{(\varepsilon)})\ge2^{1-1/b}b^{2}$ almost surely, a bound of order $\varepsilon^{-1/b}$ with $1/b<2$. This lower bound on the hole count is what gives the upper bound on the Euler characteristic in the title of \cite{Ho14}, through the identity $\chi(\gamma)=1-N_{\gamma}[0,\infty)$ for the sausage of radius $\gamma$ \cite[p.~332]{Ho14}; for each fixed radius, this identity holds almost surely (Remark~\ref{rem:euler}). The boundary case $b=1/2$, in which the radius of the sausage is comparable to the linear scale of the area cutoff, is excluded, and Honzl notes that the passage from components above an area threshold to all components of the complement of the sausage remains open \cite[p.~332]{Ho14}. Honzl's theorem is not used in our proofs. An area cutoff proportional to $\varepsilon^{2}$ has this boundary scaling, and Brownian scaling turns a fixed radius $r$ at time $T$ into the radius $\varepsilon=r/\sqrt{T}$ at unit time; but the total Betti count has no area cutoff. It counts all holes of $K_{1}^{(\varepsilon)}$, including holes of area much smaller than $\varepsilon^{2}$, which no count above an area cutoff controls, so that even a boundary version of Honzl's law would not by itself determine the Betti number. Theorem~\ref{thm:B} gives a lower bound for the total hole count at the order $\varepsilon^{-2}\ell_{\varepsilon}^{-2}$, with an identified constant and without area cutoff, and Theorem~\ref{thm:C} reduces the upper bound to Hypothesis (NC), which is proved in \cite{GuII}.

\subsection{Related work and organization}\label{subsec:related}

Persistence diagrams of random objects have mostly been studied for point clouds and random complexes \cite{HST18} and, for stochastic processes, for one-dimensional sublevel-set filtrations \cite{Ba25,Pe23,Th25}. Here the input is the offset filtration of a recurrent planar path. Werner's shape law has been used by Holden, Nacu, Peres and Salisbury to study the roundness of the cavities \cite{HNPS19}; for holes of planar random walk see Bene\v{s} \cite{Be08}. The deterministic geometry of holes of planar parallel sets used here is developed in \cite{GuPS}, and the drifted case in \cite{GuDrift}.

Section~\ref{sec:setup} fixes conventions, introduces the persistence objects, and records the scaling and comparison facts. Section~\ref{sec:law} proves the almost-sure part of Theorem~\ref{thm:A}, and Section~\ref{sec:moments} its moment part. Section~\ref{sec:persistence} proves Theorem~\ref{thm:B}. Section~\ref{sec:upper} proves Theorem~\ref{thm:C} and the decomposition of the hole count. Section~\ref{sec:outlook} discusses the sharp constant and further questions. The two papers are logically independent in the following sense: the present paper proves Theorems~\ref{thm:A}--\ref{thm:C} without using \cite{GuII}, and \cite{GuII} proves Hypothesis (NC) without using the cavity law, taking from the present paper only the measurability and duality facts of Section~\ref{sec:setup}, and its Theorems~1.3 and~1.4 on the outer and splitting terms take in addition the deterministic decomposition of Proposition~\ref{prop:threeterm}; Corollary~\ref{cor:order} combines the two.

\section{Setup: persistence, cavities and scaling}\label{sec:setup}

\subsection{Offsets and homological conventions}\label{subsec:conventions}

For a nonempty compact set $A\subset\R^{2}$ write $d_{A}:=\dist(\cdot,A)$ and, for $r\ge0$, $A^{(r)}:=\{d_{A}\le r\}$, so that $\R^{2}\setminus A^{(r)}=\{d_{A}>r\}$. Coefficients are taken in a fixed field $k$. Betti numbers of compact sets are ranks of \v{C}ech cohomology, with values in $\{0,1,\dots\}\cup\{\infty\}$; by Alexander duality, $\beta_{1}(A^{(r)})$ is the number of bounded connected components of $\{d_{A}>r\}$, the \emph{holes} of $A^{(r)}$, with no regularity assumption \cite[Remark~2.1]{GuPS}. A radius $r>0$ is \emph{regular} if it is not a critical value of $d_{A}$ in the sense of Fu \cite{Fu85}. At a regular radius, $A^{(r)}$ is a compact topological $2$-manifold with Lipschitz boundary, and its singular and \v{C}ech Betti numbers coincide \cite[Lemma~3.3]{GuPS}. The critical values of a compact planar set form a Lebesgue-null set \cite[Theorem~4.1]{Fu85}, and for the Brownian trace every fixed $r>0$ is almost surely regular \cite[Theorem~4.1]{RSS09}.

We distinguish two quantifiers. ``For each fixed $r$, almost surely'' allows the null set to depend on $r$; regularity, finiteness of $\beta_{1}$, and the agreement of singular and \v{C}ech numbers are asserted for the Brownian trace only in this form. ``Almost surely, for all $r$'' uses one null set; it is available for the pathwise \v{C}ech comparisons and lower bounds below. For each fixed $T,r>0$, $\E[\beta_{1}(K_{T}^{(r)})]<\infty$ \cite[Proposition~4.1 and Corollary~4.2]{RSM09}.

\begin{remark}[Euler characteristic]\label{rem:euler}
Fix $T,r>0$. Almost surely, $r$ is regular for $d_{K_{T}}$ and $\beta_{1}^{T}(r)<\infty$; on this event $K_{T}^{(r)}$ is a compact connected $2$-manifold with boundary in the plane, its homology is finitely generated and vanishes in degrees at least $2$, and $\chi(K_{T}^{(r)})=1-\beta_{1}^{T}(r)$. Hence $\E[\chi(K_{T}^{(r)})]=1-\E[\beta_{1}^{T}(r)]$. The identity is asserted for each fixed pair $(T,r)$; at a random critical radius the \v{C}ech number might be infinite, and no extended Euler characteristic is used.
\end{remark}

\subsection{Persistence diagrams and Betti-curve observables}\label{subsec:persistence}

The degree-one persistence module of the offset filtration $(A^{(r)})_{r\ge0}$ is $r\mapsto H_{1}(A^{(r)};k)$, with singular homology and the maps induced by inclusions. It is $q$-tame: for $0\le u<v$ and a finite set $P\subset A$ at Hausdorff distance $\delta<v-u$ from $A$, one has $A^{(u)}\subset P^{(u+\delta)}\subset A^{(v)}$, and the middle set is a finite union of discs, so the map $H_{1}(A^{(u)})\to H_{1}(A^{(v)})$ has finite rank. Hence the module has a persistence diagram $\Dgm_{1}(A)$ \cite{CDGO16}, a multiset in $\Delta:=\{(b,d):0\le b<d<\infty\}$, locally finite away from the diagonal. Deaths are finite: by Jung's theorem $A$ lies in a closed disc of radius $\diam A/\sqrt{3}$, whose centre lies in every disc $\overline{B}(a,r)$, $a\in A$, once $r\ge\diam A/\sqrt{3}$, so that $A^{(r)}$ is then star-shaped and $H_{1}(A^{(r)})=0$. (Degree zero, whose essential class has infinite death time, is not used.) The map $A\mapsto d_{A}$ is continuous for the Hausdorff metric and the topology of local uniform convergence, so the stability theorem \cite{CEH07,CDGO16} makes $\Dgm_{1}(K_{T})$ a random diagram. Since \v{C}ech cohomology is contravariant, we do not identify the two theories as functors; we compare only ranks at regular radii. With $\mu_{A}$ the counting measure of $\Dgm_{1}(A)$, the alive-count formula
\[
\beta_{1}\bigl(A^{(r)}\bigr)=\mu_{A}\bigl(\{(b,d)\in\Delta:b\le r<d\}\bigr)
\]
holds at every regular $r$, hence for almost every $r$, and, for $A=K_{T}$ and each fixed $r$, almost surely. Indeed, for every $r_{*}>0$ the set of critical points $x$ of $d_{A}$ with $d_{A}(x)\ge r_{*}$ is compact, being closed and contained in the convex hull of $A$, so its image under the continuous function $d_{A}$, the set of critical values in $[r_{*},\infty)$, is closed; taking $r_{*}=r/2$, a regular $r>0$ has a neighbourhood $[r-\delta,r+\delta]$ free of critical values, and by the isotopy lemma for distance functions \cite{Gr93,CCL09} the inclusions $A^{(s)}\subset A^{(s')}$, $r-\delta\le s\le s'\le r+\delta$, are homotopy equivalences, so the module is constant near $r$ and the number of intervals containing $r$ is $\dim H_{1}(A^{(r)})$.

We write $\beta_{1}^{T}(r):=\beta_{1}(K_{T}^{(r)})$. For a bounded Borel $\psi$ supported in $[r_{0},r_{1}]\subset(0,\infty)$ let $\phi_{\psi}(b,d):=\int_{b}^{d}\psi(r)\,dr$ and $\Phi_{\psi}(T):=\int_{\Delta}\phi_{\psi}\,d\mu_{K_{T}}$. By the alive-count formula and Tonelli's theorem,
\begin{equation}\label{eq:Phi}
\Phi_{\psi}(T)=\int_{r_{0}}^{r_{1}}\beta_{1}^{T}(r)\,\psi(r)\,dr
\end{equation}
for $\psi\ge0$, and for signed $\psi$ by Fubini's theorem, which applies because of the following integrated bound \cite[Theorem~1.1]{GuPS}: for every compact $A\subset\R^{2}$ and $r_{0}>0$,
\begin{equation}\label{eq:integrated}
\int_{r_{0}}^{\infty}\beta_{1}\bigl(A^{(r)}\bigr)\,dr\le4050\,\frac{(\diam A)^{4}}{r_{0}^{3}} .
\end{equation}
Since $\diam K_{T}\le2\sup_{t\le T}|B_{t}|$ and $\E[(\diam K_{T})^{4}]=T^{2}\E[(\diam K_{1})^{4}]<\infty$, this gives
\[
\E|\Phi_{\psi}(T)|\le4050\,\norm{\psi}_{\infty}\,\E\bigl[(\diam K_{1})^{4}\bigr]\,T^{2}r_{0}^{-3}.
\] The bound \eqref{eq:integrated} serves only for integrability: it is far from the order of magnitude, and the integration over the radius cannot be removed, as Remark~\ref{rem:nodet} below shows. The joint measurability of $(r,\omega)\mapsto\beta_{1}^{T}(r)$, needed for Tonelli's theorem in the radius variable, is Lemma~\ref{lem:measholes}.

\subsection{Cavities and holes}\label{subsec:cavities}

Let $A\subset\R^{2}$ be compact. Its \emph{cavities} are the bounded connected components of $\R^{2}\setminus A$; they lie in the convex hull of $A$, since the complement of the convex hull is connected and unbounded. The inradius of a cavity $U$ relative to $A$ is $\rho_{A}(U):=\sup_{x\in U}d_{A}(x)$, and for $r>0$ we set
\[
N_{A}(r):=\card\{U\text{ cavity of }A:\rho_{A}(U)>r\}.
\]
If $\rho_{A}(U)>r$ and $d_{A}(x)>r$ with $x\in U$, the open disc $B(x,d_{A}(x))$ is a connected subset of $\R^{2}\setminus A$, hence lies in $U$; so $U$ contains a disc of radius larger than $r$, and $N_{A}(r)<\infty$ by packing in the convex hull.

\begin{lemma}[Cavity comparison and scaling]\label{lem:comparison}
For every compact $A\subset\R^{2}$ and $r>0$,
\[
N_{A}(r)\le\beta_{1}\bigl(A^{(r)}\bigr),
\]
and for every $a>0$, $N_{aA}(r)=N_{A}(r/a)$ and $\beta_{1}((aA)^{(r)})=\beta_{1}(A^{(r/a)})$.
\end{lemma}

\begin{proof}
Let $U$ be a cavity with $\rho_{A}(U)>r$, and pick $x_{U}\in U$ with $d_{A}(x_{U})>r$. The component $V_{U}$ of $\{d_{A}>r\}$ containing $x_{U}$ is a connected subset of $\R^{2}\setminus A$ meeting $U$, hence $V_{U}\subset U$ is bounded, that is, a hole of $A^{(r)}$; distinct cavities are disjoint and give distinct holes. The dilation $x\mapsto ax$ maps the components of $\R^{2}\setminus A$ and of $\{d_{A}>r/a\}$ bijectively onto those of $\R^{2}\setminus aA$ and of $\{d_{aA}>r\}$, preserving boundedness, and $d_{aA}(ax)=a\,d_{A}(x)$.
\end{proof}

The inequality is strict in general (Fig.~\ref{fig:cavities}): a cavity may split into several holes when a narrow neck closes, and the unbounded component of $\R^{2}\setminus A$ may contribute bounded pockets, as for a nearly closed circular arc. These two sources of holes are analysed in Section~\ref{subsec:decomposition}.

\begin{figure}[H]
\centering
\includegraphics[width=\textwidth]{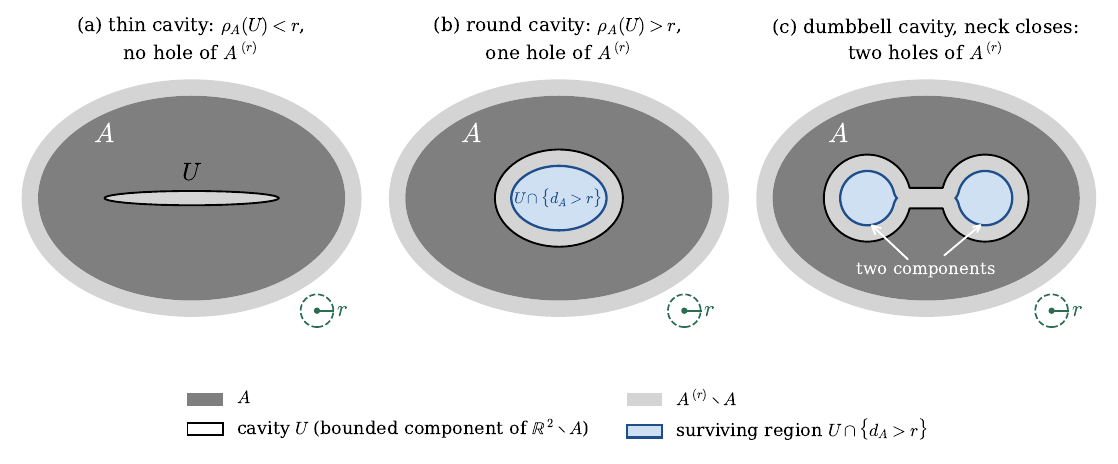}
\caption{\textbf{Cavities and holes at radius $r$.} The set $A$ (dark grey) has one cavity $U$, outlined in black; the added part $A^{(r)}\setminus A$ of the offset is light grey, and the surviving region $U\cap\{d_{A}>r\}$ is light blue. The dashed circle has radius $r$. The surviving region may be empty (a, thin cavity, $\rho_{A}(U)<r$), connected (b, round cavity), or disconnected (c, dumbbell cavity whose neck closes). Hence $N_{A}(r)\le\beta_{1}(A^{(r)})$, with strict inequality in general. In the notation of Section~\ref{subsec:decomposition}, panel (c) is a cavity with $m_{i}=2$, which contributes one hole to the one-pocket term and one to the splitting term.}\label{fig:cavities}
\end{figure}

We work on the canonical Wiener space $\Omega=C([0,\infty);\R^{2})$ with its Borel $\sigma$-field, and write $K_{T}(w)=w[0,T]$. The Hausdorff distance satisfies $d_{H}(K_{T}(w),K_{T}(w'))\le\sup_{t\le T}|w(t)-w'(t)|$, so for every nonempty compact $S\subset\R^{2}$ the map $w\mapsto\dist(S,K_{T}(w))$ is continuous.

\begin{lemma}[Measurability of hole and cavity counts]\label{lem:measholes}
Fix $T>0$. For $p,q\in\R^{2}$ and $r\ge0$, let $\Sigma_{r}(p,q)$ be the event that $p$ and $q$ lie in the same component of $\{d_{K_{T}}>r\}$ (for $r=0$, of $\R^{2}\setminus K_{T}$). Then $\{(r,w):w\in\Sigma_{r}(p,q)\}$ is Borel in $[0,\infty)\times\Omega$, and the maps $(r,w)\mapsto\beta_{1}^{T}(r)$, $(r,w)\mapsto N_{K_{T}}(r)$ and $(r,w)\mapsto\Ncrit_{K_{T}}(r)$, the last being the number of holes $U$ of $K_{T}^{(r)}$ with $\max_{\overline U}d_{K_{T}}\le2r$, are jointly Borel on $(0,\infty)\times\Omega$, with values in $\{0,1,\dots\}\cup\{\infty\}$.
\end{lemma}

\begin{proof}
We claim that $\Sigma_{r}(p,q)$ holds if and only if there are a polygonal path $\Gamma$ from $p$ to $q$ with rational interior vertices and a rational $\delta>0$ with $\dist(\Gamma,K_{T})\ge r+\delta$; since each set $\{(r,w):\dist(\Gamma,K_{T}(w))\ge r+\delta\}$ is closed, this proves the first assertion. The condition is clearly sufficient. Conversely, the component of $p$ is open and connected, hence polygonally connected; a polygonal path $\Gamma_{0}$ in it from $p$ to $q$ is compact, so $\delta_{0}:=\dist(\Gamma_{0},K_{T})-r>0$, and moving each interior vertex by less than $\delta_{0}/2$ to a rational point moves every segment by less than $\delta_{0}/2$ in the Hausdorff distance, since $|((1-t)a'+tb')-((1-t)a+tb)|\le\max(|a-a'|,|b-b'|)$. Every component of an open set contains rational points. The component of a rational $p$ with $d_{K_{T}}(p)>r$ is unbounded if and only if for every $n$ there is a rational $q$ with $|q|>n$ and $\Sigma_{r}(p,q)$; it has depth larger than $2r$ if and only if $\Sigma_{r}(p,q')$ holds for some rational $q'$ with $d_{K_{T}}(q')>2r$, because $d_{K_{T}}$ is continuous. Hence $\{\beta_{1}^{T}(r)\ge k\}$ is the union over $k$-tuples of distinct rational points $p_{1},\dots,p_{k}$ of the events that each $p_{j}$ satisfies $d_{K_{T}}(p_{j})>r$ and lies in a bounded component of $\{d_{K_{T}}>r\}$, and that $\Sigma_{r}(p_{j},p_{j'})$ fails for $j\ne j'$; this is Borel in $(r,w)$. Adding the depth condition gives $\Ncrit_{K_{T}}$. For $N_{K_{T}}(r)$ one requires instead that the $p_{j}$ satisfy $d_{K_{T}}(p_{j})>r$ and lie in pairwise distinct bounded components of $\R^{2}\setminus K_{T}$, that is, one uses $\Sigma_{0}$ in place of $\Sigma_{r}$. Requiring in addition that the component of $p_{j}$ in $\R^{2}\setminus K_{T}$ be bounded, or unbounded, which is tested with $\Sigma_{0}$ in the same way, shows that the numbers of holes of $K_{T}^{(r)}$ lying in cavities of $K_{T}$, and in the unbounded component of $\R^{2}\setminus K_{T}$, are jointly Borel as well.
\end{proof}

\subsection{Brownian scaling}\label{subsec:scaling}

For $T>0$ the process $W^{T}(t):=T^{-1/2}B(Tt)$ is a standard planar Brownian motion, and pathwise $K_{T}=\sqrt{T}\,W^{T}[0,1]$, hence $K_{T}^{(r)}=\sqrt{T}\,(W^{T}[0,1])^{(r/\sqrt{T})}$. These dilations commute with the inclusions of the filtration, so they identify the persistence modules after reparametrization of the radius, and $\Dgm_{1}(K_{T})\eqd\{(\sqrt{T}b,\sqrt{T}d):(b,d)\in\Dgm_{1}(K_{1})\}$. We use scaling in the following form.

\begin{lemma}[Transfer from unit time]\label{lem:transfer}
Let $T>0$ and $\varepsilon:=T^{-1/2}$. Pathwise, for all $r>0$, $\beta_{1}^{T}(r)=\beta_{1}\bigl((W^{T}[0,1])^{(\varepsilon r)}\bigr)$ and $N_{K_{T}}(r)=N_{W^{T}[0,1]}(\varepsilon r)$, and $W^{T}$ has the law of $B$. Hence, if $F:\Omega\to[-\infty,\infty]$ is Borel and $F(w)$ is computed from the radius-indexed families $(\beta_{1}((w[0,1])^{(\varepsilon r)}))_{r>0}$ and $(N_{w[0,1]}(\varepsilon r))_{r>0}$ (for instance a radius integral against a bounded weight, which is Borel by Lemma~\ref{lem:measholes} and Tonelli's theorem; for a signed weight it converges absolutely for every path, by \eqref{eq:integrated}), then the same functional computed from $(\beta_{1}^{T}(r))_{r>0}$ and $(N_{K_{T}}(r))_{r>0}$ equals $F(W^{T})$, and has the law of $F(B)$. Moreover
\[
\frac{(\log T)^{2}}{T}=4\,\varepsilon^{2}\ell_{\varepsilon}^{2},\qquad
\sup_{r\in[r_{0},r_{1}]}\Bigl|\frac{\ell_{\varepsilon}}{\ell_{\varepsilon r}}-1\Bigr|\longrightarrow0\quad(\varepsilon\downarrow0)
\]
for every compact $[r_{0},r_{1}]\subset(0,\infty)$. Consequently, if such functionals $F_{\varepsilon}$ satisfy $F_{\varepsilon}(B)\to c$ in probability, or in $L^{p}$, as $\varepsilon\downarrow0$, for a constant $c$, then the corresponding functionals at time $T=\varepsilon^{-2}$ converge to $c$ in the same sense as $T\to\infty$; and their expectations and probabilities coincide.
\end{lemma}

\begin{proof}
The pathwise identities follow from $K_{T}=\sqrt{T}\,W^{T}[0,1]$ and Lemma~\ref{lem:comparison} with $a=\sqrt{T}$, and $W^{T}$ is a standard Brownian motion, so $F(W^{T})\eqd F(B)$. The normalization identity is $\log T=2\ell_{\varepsilon}$, and $\ell_{\varepsilon r}=\ell_{\varepsilon}-\log r$. Convergence in probability or in $L^{p}$ to a constant depends only on the one-dimensional laws.
\end{proof}

Almost-sure statements do not transfer: equality in law at each fixed $T$ gives no information on the joint behaviour in $T$ of the original path. Scaling also forces the radius profile of any sharp asymptotic.

\begin{proposition}[Radius profile]\label{prop:profile}
Let $a(T)=T\varsigma(T)$ with $\varsigma$ slowly varying and $a(T)>0$ for large $T$, and suppose that $\E[\beta_{1}^{T}(r)]/a(T)\to f(r)\in[0,\infty)$ for every $r>0$. Then $f(r)=\kappa/r^{2}$ for a constant $\kappa\ge0$, which depends on the normalization (replacing $a$ by $\lambda a$ replaces $\kappa$ by $\kappa/\lambda$) but not on $r$.
\end{proposition}

\begin{proof}
By Lemma~\ref{lem:transfer}, $\E[\beta_{1}^{T}(r)]=\E[\beta_{1}^{T(s/r)^{2}}(s)]$ for $r,s>0$. Dividing by $a(T)$ and writing $a(T(s/r)^{2})/a(T)=(s/r)^{2}\varsigma(T(s/r)^{2})/\varsigma(T)\to(s/r)^{2}$ gives $f(r)=f(s)s^{2}/r^{2}$.
\end{proof}

\section{The cavity inradius law}\label{sec:law}

In this section and the next, $K_{1}=B[0,1]$, the cavities $(A_{i})_{i\ge1}$ of $K_{1}$ are ranked by nonincreasing area $L_{i}$, $\rho_{i}=\rho_{K_{1}}(A_{i})$, $G_{i}:=G(A_{i})$, and $\Nin=N_{K_{1}}$. Since $\partial A_{i}\subset K_{1}$, one has $d_{K_{1}}=\dist(\cdot,\partial A_{i})$ on $A_{i}$, hence
\begin{equation}\label{eq:exact}
\rho_{i}=G_{i}\sqrt{L_{i}},\qquad
\Nin(\eta)=\card\{i\ge1:\rho_{i}>\eta\}=\card\bigl\{i\ge1:G_{i}>\eta/\sqrt{L_{i}}\bigr\}.
\end{equation}
The cavities are simply connected, being bounded components of the complement of a compact connected set. Almost surely there are infinitely many of them, by Input~I below; on the null complementary event the ranked sequences are padded with $L_{i}=\rho_{i}=G_{i}=0$.

\subsection{The inputs}\label{subsec:inputs}

\noindent\textbf{Input I (Le Gall's law).} \emph{Almost surely,}
\begin{equation}\label{eq:legall}
u\ell_{u}^{2}\,M(u)\longrightarrow\gamma:=2\pi\qquad(u\downarrow0).
\end{equation}
\emph{This is due to Le Gall \cite{LG91,LG92}; see also \cite{HNPS19}. Le Gall counts cavities of area at least $u$; since $M$ is monotone, the asymptotics are the same.} Most arguments below use only that $\gamma\in(0,\infty)$ is deterministic.

\medskip
\noindent\textbf{Input II (Werner's empirical shape law).} \emph{Let $\mathscr{C}$ be the set of bounded simply connected open sets containing the origin, with Werner's $\sigma$-field $\mathcal{T}$ \cite[\S2.1]{We94}, let $\mathscr{C}'$ be the space of shapes, the quotient of $\mathscr{C}$ by translations and dilations with the induced $\sigma$-field \cite[\S4.1]{We94}, and let $\shape(A)\in\mathscr{C}'$ be the shape of $A$. There is a probability measure $\cL$ on $\mathscr{C}'$ such that, for every bounded measurable $H:\mathscr{C}'\to\R$, almost surely}
\[
\frac{1}{n}\sum_{i=1}^{n}H\bigl(\shape(A_{i})\bigr)\longrightarrow\int H\,d\cL\qquad(n\to\infty).
\]
\emph{This is \cite[Th\'eor\`eme~14]{We94}: ``pour toute fonction $F$ mesurable born\'ee de $\mathscr{C}'$ dans $\R$ on a p.s.\ $(F(K_{1})+\dots+F(K_{n}))/n\to\widetilde F$'', with $K_{n}$ the shape of the $n$-th largest cavity and $\widetilde F=\int F\,d\cL$.} The null set may depend on $H$. We apply Input~II only to countably many indicators $H=\mathbf{1}\{G>q\}$, and write $\Omega_{q}$ for the event of probability one on which
\begin{equation}\label{eq:werner}
\frac{1}{n}\card\{1\le i\le n:G_{i}>q\}\longrightarrow F(q):=\cL(G>q)\qquad(n\to\infty).
\end{equation}

The sets $\{G>q\}$ are measurable. The $\sigma$-field $\mathcal{T}$ is generated by the indicators of $\{R\subset C\}$ for the countable family of rational rectangles $R$ \cite[\S2.1.2]{We94}. An open disc is the union of the closed rational rectangles it contains, so $\{\mathscr{D}(x,r)\subset C\}$ is measurable for rational $x$ and $r$, and so are $\dist(x,\R^{2}\setminus C)=\sup\{r\in\mathbb{Q}_{>0}:\mathscr{D}(x,r)\subset C\}$, the inradius $\sup_{x\in\mathbb{Q}^{2}\cap C}\dist(x,\R^{2}\setminus C)$, and the area $\sup\{\Leb_{2}(R_{1}\cup\dots\cup R_{k}):R_{j}\subset C\text{ rational rectangles}\}$. The inradius and the square root of the area are translation invariant and $1$-homogeneous, so $G$ descends to a measurable function on $\mathscr{C}'$. The tail $F$ is nonincreasing, bounded by $1$, and vanishes on $[\pi^{-1/2},\infty)$, and by Tonelli's theorem
\begin{equation}\label{eq:layer}
\int G^{2}\,d\cL=\int_{0}^{\infty}F\bigl(\sqrt{w}\bigr)\,dw .
\end{equation}
No continuity of $F$, that is, no assumption on atoms of the law of $G$ under $\cL$, is needed below.

\subsection{Ranking, measurability and ties}\label{subsec:ranking}

Fix an enumeration $(\xi_{n})_{n\ge1}$ of $\mathbb{Q}^{2}$ and write $\Sigma:=\Sigma_{0}$ in the notation of Lemma~\ref{lem:measholes} with $T=1$.

\begin{lemma}[Ranked cavity functionals]\label{lem:ranked}
Ranking the cavities by nonincreasing area and breaking ties by the smallest index of a rational point they contain, the sequences $(L_{i})$, $(\rho_{i})$ and $(G_{i})$ are random variables on $\Omega$, $\Nin(\eta)=\card\{i:\rho_{i}>\eta\}$ pointwise, and all almost-sure convergence events used below are measurable.
\end{lemma}

\begin{proof}
The event $\Sigma_{\infty}(q)$ that $q$ lies in the unbounded component of $\R^{2}\setminus K_{1}$ is the union over rational $R$ of $\{\max_{t\le1}|B_{t}|<R\}\cap\Sigma(q,(\lceil R\rceil+1,0))$, since $\{|z|\ge R\}$ is then a connected subset of that component. The Borel events
\[
E_{n}:=\{\xi_{n}\notin K_{1}\}\cap\Sigma_{\infty}(\xi_{n})^{c}\cap\bigcap_{m<n}\bigl(\{\xi_{m}\in K_{1}\}\cup\Sigma(\xi_{m},\xi_{n})^{c}\bigr)
\]
say that $\xi_{n}$ lies in a cavity containing no rational point of smaller index, so $n\mapsto$ (cavity of $\xi_{n}$) is a bijection from $\{n:E_{n}\}$ onto the set of cavities. The set $\{(x,w):\Sigma(x,\xi_{n})(w)\}$ is Borel, since $x$ lies in the component of $\xi_{n}$ if and only if there are rational $q$ and $r>0$ with $|x-q|<r$, $d_{K_{1}}(q)>r$ and $\Sigma(q,\xi_{n})$. Hence, by Tonelli's theorem, the area $\Lambda_{n}:=\int\mathbf{1}\Sigma(x,\xi_{n})\,dx$ of the cavity of $\xi_{n}$ is a random variable on $E_{n}$, and so is its inradius $P_{n}:=\sup_{m}d_{K_{1}}(\xi_{m})\mathbf{1}\Sigma(\xi_{m},\xi_{n})$. Then $M(u)=\sum_{n}\mathbf{1}_{E_{n}}\mathbf{1}\{\Lambda_{n}>u\}$ is finite, since the cavities lie in the disc of radius $\max_{t\le1}|B_{t}|$, and measurable; the ranked areas satisfy $\{L_{i}>u\}=\{M(u)\ge i\}$. Define recursively $\sigma(i)$ as the smallest $n\notin\{\sigma(1),\dots,\sigma(i-1)\}$ with $E_{n}$ and $\Lambda_{n}=L_{i}$; each $\{\sigma(i)=n\}$ is Borel by induction, $i\mapsto\sigma(i)$ is a bijection onto $\{n:E_{n}\}$, and $\rho_{i}:=P_{\sigma(i)}$, $G_{i}:=\rho_{i}/\sqrt{L_{i}}$ are random variables with $\card\{i:\rho_{i}>\eta\}=\Nin(\eta)$.
\end{proof}

\begin{remark}[Ties are immaterial]\label{rem:ties}
Werner numbers the cavities by decreasing area without discussing ties. Two numberings consistent with nonincreasing area differ by permutations inside groups of equal area, and on the event of Input~I such a group containing rank $n$ has size $o(n)$: with $\zeta_{i}$ as in Lemma~\ref{lem:rankarea}, whose proof uses only $M$, $L_{n}=L_{n+k}$ forces $\zeta_{n+k}/\zeta_{n}\ge1+k/n$ while $\zeta_{i}\to1$, and the smallest rank of the group tends to infinity because each positive area is attained finitely often. Hence the Ces\`aro limits in Input~II do not depend on the numbering, and Input~II holds for the numbering of Lemma~\ref{lem:ranked}.
\end{remark}

\subsection{Deterministic preliminaries}\label{subsec:prelim}

Set
\[
a_{\eta}:=\frac{\gamma}{4\eta^{2}\ell_{\eta}^{2}}\quad(0<\eta<1),\qquad
\zeta_{i}:=\frac{i(\log i)^{2}L_{i}}{\gamma}\quad(i\ge3),\qquad
\tau_{i,\eta}:=\frac{\eta\sqrt{i}\,\log i}{\sqrt{\gamma}}\quad(i\ge3),
\]
so that the almost-sure statement of Theorem~\ref{thm:A} reads $\Nin(\eta)/a_{\eta}\to\int G^{2}d\cL$, since $\gamma/4=\pi/2$, and
\begin{equation}\label{eq:tauzeta}
\frac{\eta}{\sqrt{L_{i}}}=\frac{\tau_{i,\eta}}{\sqrt{\zeta_{i}}},\qquad i\ge3 .
\end{equation}
For $\delta\in(0,1/10]$ let $n_{m}:=\lceil(1+\delta)^{m}\rceil$, $I_{m}:=\{i:n_{m}\le i<n_{m+1}\}$, $Q_{\delta}:=\{(1+\delta)^{k}:k\in\mathbb{Z}\}$, and, when $n_{m+1}>n_{m}$,
\[
F_{m}(q):=\frac{\card\{i\in I_{m}:G_{i}>q\}}{n_{m+1}-n_{m}} .
\]

\begin{lemma}[Rank--area law]\label{lem:rankarea}
Almost surely, $\zeta_{i}\to1$.
\end{lemma}

This ranked-area asymptotic is \cite[Lemma~3.1]{HNPS19}; we include the short proof for completeness.

\begin{proof}
Fix $\delta\in(0,1)$ and let $v_{i}:=(1+\delta)\gamma/(i(\log i)^{2})$, $w_{i}:=(1-\delta)\gamma/(i(\log i)^{2})$. Since $\log(1/v_{i})\sim\log(1/w_{i})\sim\log i$, Input~I gives $M(v_{i})=i(1+o(1))/(1+\delta)$ and $M(w_{i})=i(1+o(1))/(1-\delta)$. Eventually $M(v_{i})<i$, which excludes $L_{i}>v_{i}$, and $M(w_{i})\ge i+1$, which gives $L_{i}\ge L_{i+1}>w_{i}$. Thus $1-\delta<\zeta_{i}\le1+\delta$ eventually; let $\delta\downarrow0$ along a sequence.
\end{proof}

\begin{lemma}[Block empirical tails]\label{lem:blocktails}
Fix $\delta\in(0,1/10]$ and $q>0$. On $\Omega_{q}$, $n_{m+1}>n_{m}$ for all large $m$ and $F_{m}(q)\to F(q)$ as $m\to\infty$.
\end{lemma}

\begin{proof}
On $\Omega_{q}$, $C_{n}:=\card\{i\le n:G_{i}>q\}=nF(q)+n\alpha_{n}$ with $\alpha_{n}\to0$. For $m\ge m_{0}(\delta):=\lceil\log(2/\delta)/\log(1+\delta)\rceil$ one has $n_{m+1}-n_{m}\ge\delta(1+\delta)^{m}-1\ge\frac{\delta}{2}(1+\delta)^{m}$ and $n_{m+1}\le2(1+\delta)^{m+1}$, so $n_{m+1}/(n_{m+1}-n_{m})\le c(\delta):=4(1+\delta)/\delta$. Since $\card\{i\in I_{m}:G_{i}>q\}=C_{n_{m+1}-1}-C_{n_{m}-1}$,
\[
|F_{m}(q)-F(q)|\le c(\delta)\bigl(|\alpha_{n_{m+1}-1}|+|\alpha_{n_{m}-1}|\bigr)\longrightarrow0 .\qedhere
\]
\end{proof}

\begin{lemma}[Threshold calibration]\label{lem:calib}
Let $\delta\in(0,1/10]$.
\begin{itemize}
\item[(a)] For every $\eta>0$, $i\mapsto\tau_{i,\eta}$ is nondecreasing on $\{i\ge3\}$.
\item[(b)] $\tau_{i,\eta}\sqrt{a_{\eta}/i}=\log i/(2\ell_{\eta})$; consequently, for every $v_{0}\in(0,1)$, $\sup\{|\tau_{i,\eta}/\sqrt{i/a_{\eta}}-1|:i\ge3,\ v_{0}/2\le i/a_{\eta}\le4\}\to0$ as $\eta\downarrow0$.
\item[(c)] There is $m_{1}(\delta)$ such that $\tau_{n_{m},\eta}\le\tau_{i,\eta}\le(1+\delta)\tau_{n_{m},\eta}$ for all $m\ge m_{1}(\delta)$, $i\in I_{m}$ and $\eta>0$.
\end{itemize}
\end{lemma}

\begin{proof}
(a) The derivative of $\sqrt{x}\log x$ is $(\log x+2)/(2\sqrt{x})>0$ on $[1,\infty)$. (b) The identity is direct. If $i/a_{\eta}\in[v_{0}/2,4]$, then $\log i=\log a_{\eta}+\log(i/a_{\eta})$ with $\log a_{\eta}=\log(\gamma/4)+2\ell_{\eta}-2\log\ell_{\eta}$, so $|\log i/(2\ell_{\eta})-1|\le(|\log(\gamma/4)|+2\log\ell_{\eta}+\log(8/v_{0}))/(2\ell_{\eta})\to0$ uniformly. (c) By (a), and since $\tau_{i,\eta}/\tau_{n_{m},\eta}\le\sqrt{n_{m+1}/n_{m}}\,\log n_{m+1}/\log n_{m}\to\sqrt{1+\delta}<1+\delta$, independently of $\eta$.
\end{proof}

\begin{lemma}[Riemann bounds]\label{lem:riemann}
Let $\delta\in(0,1/10]$. There is $m_{2}(\delta)\ge m_{0}(\delta)$ such that $n_{m}>n_{m-1}$ and $n_{m+1}-n_{m}\le(1+\delta)^{2}(n_{m}-n_{m-1})$ for $m\ge m_{2}(\delta)$. Consequently, for every nonincreasing $g:(0,\infty)\to[0,1]$, every $a>0$ and all integers $m_{2}(\delta)<M_{1}\le M_{2}$, with $v_{m}:=n_{m}/a$,
\begin{equation}\label{eq:riemann}
a\int_{v_{M_{1}}}^{v_{M_{2}+1}}g\ \le\ \sum_{m=M_{1}}^{M_{2}}(n_{m+1}-n_{m})\,g(v_{m})\ \le\ (1+\delta)^{2}\,a\int_{v_{M_{1}-1}}^{\infty}g .
\end{equation}
\end{lemma}

\begin{proof}
The gap ratio $(n_{m+1}-n_{m})/(n_{m}-n_{m-1})$ tends to $1+\delta<(1+\delta)^{2}$. For the upper bound, $(n_{m+1}-n_{m})g(v_{m})\le(1+\delta)^{2}a(v_{m}-v_{m-1})g(v_{m})\le(1+\delta)^{2}a\int_{v_{m-1}}^{v_{m}}g$ by monotonicity; sum over $m$. For the lower bound, $(n_{m+1}-n_{m})g(v_{m})=a(v_{m+1}-v_{m})g(v_{m})\ge a\int_{v_{m}}^{v_{m+1}}g$.
\end{proof}

\subsection{Proof of the almost-sure law}\label{subsec:proofA}

\begin{proposition}\label{prop:as}
Almost surely, $\Nin(\eta)/a_{\eta}\to\int G^{2}\,d\cL$ as $\eta\downarrow0$. Moreover $0<\cin\le1/2$.
\end{proposition}

\begin{proof}
Let $\mathcal{D}:=\{1/k:k\ge10\}$, and let $\Omega_{0}$ be the intersection of the event of Input~I with the events $\Omega_{q}$, $q\in Q_{\delta}$, $\delta\in\mathcal{D}$; it is a countable intersection, so $\Prob(\Omega_{0})=1$, and on $\Omega_{0}$, $\zeta_{i}\to1$ and $F_{m}(q)\to F(q)$ for all these $q$ (Lemmas~\ref{lem:rankarea} and~\ref{lem:blocktails}). Fix $\omega\in\Omega_{0}$ and parameters $\delta\in\mathcal{D}$, $\epsilon\in\{1/k:k\ge2\}$, $v_{0}\in\{1/k:k\ge5\}$. Put $q_{0}:=(1+\delta)^{-3}\sqrt{v_{0}}$ and let $\Theta:=Q_{\delta}\cap[q_{0},2]$, a finite set. Choose $i_{*}$ with $(1+\delta)^{-1}\le\zeta_{i}\le1+\delta$ for $i\ge i_{*}$, and $M_{*}$ with $|F_{m}(q)-F(q)|\le\epsilon$ for $m\ge M_{*}$ and $q\in\Theta$. For $\eta\in(0,1)$ let
\[
m_{-}:=\min\{m:n_{m}\ge v_{0}a_{\eta}\},\qquad m_{\dagger}:=\max\{m:n_{m}\le2a_{\eta}\}.
\]
Since $a_{\eta}\to\infty$, there is $\eta_{0}>0$ such that for $\eta\le\eta_{0}$: (i) $m_{-}-1\ge\max(m_{1}(\delta),m_{2}(\delta),M_{*})$ and $n_{m_{-}}\ge\max(i_{*},3)$; (ii) the supremum in Lemma~\ref{lem:calib}(b) is at most $\delta$; (iii) $a_{\eta}\ge\max(4,2/v_{0})$, $n_{m_{-}}\le2v_{0}a_{\eta}$ and $n_{m_{\dagger}+1}\le3a_{\eta}$. (By minimality, $n_{m_{-}}\le(1+\delta)n_{m_{-}-1}+1\le(1+\delta)v_{0}a_{\eta}+1$, and by maximality $n_{m_{\dagger}+1}\le(1+\delta)(2a_{\eta}+1)$; also $m_{\dagger}\ge m_{-}$.) By \eqref{eq:tauzeta}, for $i\ge\max(i_{*},3)$,
\begin{equation}\label{eq:sandwich}
(1+\delta)^{-1/2}\tau_{i,\eta}\le\eta/\sqrt{L_{i}}\le(1+\delta)^{1/2}\tau_{i,\eta}.
\end{equation}
We use the numerical facts $(1-\delta)(1+\delta)^{3/2}\ge1$, $(1+\delta)^{-1/2}(1-\delta)\sqrt2>\pi^{-1/2}$ and $(1+\delta)^{7/2}\sqrt2\le2$, valid for $\delta\le1/10$.

\emph{Step 1: high ranks.} If $i\ge n_{m_{\dagger}+1}$, then $i>2a_{\eta}$, so $i\ge i_{2}:=\lceil2a_{\eta}\rceil$ with $i_{2}/a_{\eta}\in[2,2.25]$, and by (ii), Lemma~\ref{lem:calib}(a) and \eqref{eq:sandwich}, $\eta/\sqrt{L_{i}}\ge(1+\delta)^{-1/2}(1-\delta)\sqrt2>\pi^{-1/2}\ge G_{i}$; by \eqref{eq:exact}, $\rho_{i}\le\eta$. Hence
\begin{equation}\label{eq:split}
\Nin(\eta)=\card\{i<n_{m_{-}}:\rho_{i}>\eta\}+\sum_{m=m_{-}}^{m_{\dagger}}\card\{i\in I_{m}:\rho_{i}>\eta\},
\end{equation}
and the first term is at most $n_{m_{-}}\le2v_{0}a_{\eta}$.

\emph{Step 2: upper estimate.} Let $m_{-}\le m\le m_{\dagger}$, $v_{m}:=n_{m}/a_{\eta}\in[v_{0},2]$, and $i\in I_{m}$. If $\rho_{i}>\eta$, then by \eqref{eq:exact}, \eqref{eq:sandwich}, Lemma~\ref{lem:calib}(a), (ii) and the first numerical fact, $G_{i}>(1+\delta)^{-1/2}\tau_{n_{m},\eta}\ge(1+\delta)^{-2}\sqrt{v_{m}}$. With $q_{m}:=\max\{q\in Q_{\delta}:q\le(1+\delta)^{-2}\sqrt{v_{m}}\}\in\Theta$, which satisfies $q_{m}\ge(1+\delta)^{-3}\sqrt{v_{m}}$,
\[
\card\{i\in I_{m}:\rho_{i}>\eta\}\le|I_{m}|F_{m}(q_{m})\le|I_{m}|\bigl(g_{\delta}(v_{m})+\epsilon\bigr),\qquad g_{\delta}(v):=F\bigl((1+\delta)^{-3}\sqrt{v}\bigr).
\]
By \eqref{eq:riemann} and the substitution $w=(1+\delta)^{-6}v$ in \eqref{eq:layer}, $\sum_{m}|I_{m}|g_{\delta}(v_{m})\le(1+\delta)^{2}a_{\eta}\int_{0}^{\infty}g_{\delta}=(1+\delta)^{8}a_{\eta}\int G^{2}d\cL$. The $\epsilon$-terms total at most $\epsilon\,n_{m_{\dagger}+1}\le3\epsilon a_{\eta}$. Hence $\Nin(\eta)/a_{\eta}\le2v_{0}+(1+\delta)^{8}\int G^{2}d\cL+3\epsilon$.

\emph{Step 3: lower estimate.} For $m_{-}\le m\le m_{\dagger}$ and $i\in I_{m}$, \eqref{eq:sandwich}, Lemma~\ref{lem:calib}(c) and (ii) give $\eta/\sqrt{L_{i}}\le(1+\delta)^{5/2}\sqrt{v_{m}}$. Let $p_{m}:=\min\{q\in Q_{\delta}:q\ge(1+\delta)^{5/2}\sqrt{v_{m}}\}$; then $p_{m}\le(1+\delta)^{7/2}\sqrt{v_{m}}\le2$ and $p_{m}\ge q_{0}$, so $p_{m}\in\Theta$, and $G_{i}>p_{m}$ implies $\rho_{i}>\eta$. With $h_{\delta}(v):=F((1+\delta)^{7/2}\sqrt{v})$,
\[
\card\{i\in I_{m}:\rho_{i}>\eta\}\ge|I_{m}|F_{m}(p_{m})\ge|I_{m}|\bigl(h_{\delta}(v_{m})-\epsilon\bigr).
\]
The function $h_{\delta}$ vanishes on $[1,\infty)$, $v_{m_{\dagger}+1}>2$, $v_{m_{-}}\le2v_{0}$ and $h_{\delta}\le1$, so \eqref{eq:riemann} and the substitution $w=(1+\delta)^{7}v$ give $\sum_{m}|I_{m}|h_{\delta}(v_{m})\ge a_{\eta}\bigl((1+\delta)^{-7}\int G^{2}d\cL-2v_{0}\bigr)$. Hence $\Nin(\eta)/a_{\eta}\ge(1+\delta)^{-7}\int G^{2}d\cL-2v_{0}-3\epsilon$.

\emph{Step 4.} Letting $\eta\downarrow0$ and then $\epsilon,v_{0},\delta\downarrow0$ along their countable families gives the limit at every $\omega\in\Omega_{0}$. Finally $0<G\le\pi^{-1/2}$ and $\cL$ is a probability measure, so $0<\int G^{2}d\cL\le1/\pi$ and $0<\cin\le1/2$.
\end{proof}

The value $\cin=1/2$ would require $\cL$-almost every shape to have the area of its largest inscribed disc. 

\begin{remark}[Other bounded shape functionals]\label{rem:otherH}
The proof of Proposition~\ref{prop:as} uses only that $G$ is a measurable shape functional with $0\le G\le\pi^{-1/2}$. For any measurable shape functional $H$ with $0\le H\le h<\infty$ such that the $H(A_{i})$ are random variables, the same argument, with the cut-off $2a_{\eta}$ replaced by a multiple of $a_{\eta}$ depending on $h$, gives
\[
\eta^{2}\ell_{\eta}^{2}\,\card\bigl\{i:H(A_{i})\sqrt{L_{i}}>\eta\bigr\}\ \longrightarrow\ \frac{\gamma}{4}\int H^{2}\,d\cL=\frac{\pi}{2}\int H^{2}\,d\cL,
\]
almost surely and in $L^{p}$ for every $p\in[1,2]$. Convergence in $L^{p}$ follows from the almost-sure statement as in the proof of Theorem~\ref{thm:A} in Section~\ref{sec:moments}: if $h>0$, a cavity with $H(A_{i})\sqrt{L_{i}}>\eta$ has $L_{i}>\eta^{2}/h^{2}$, so the count is at most $M(\eta^{2}/h^{2})$; with $u=\eta^{2}/h^{2}$ one has $\ell_{\eta}^{2}/\ell_{u}^{2}\to\frac14$, so the normalized count is at most $h^{2}u\ell_{u}^{2}M(u)$ for small $\eta$, whose squares are uniformly integrable; the case $h=0$ is trivial. An example, provided it is measurable with respect to the $\sigma$-field of \cite{We94} on shapes, which we do not verify here, is the normalized maximal conformal radius $H(A):=\sup_{z\in A}\operatorname{rad}(z,A)/\sqrt{\Leb_{2}(A)}$, where $\operatorname{rad}(z,A)$ is the conformal radius of the simply connected domain $A$ at $z$; by the Schwarz lemma and Koebe's one-quarter theorem, $\dist(z,\partial A)\le\operatorname{rad}(z,A)\le4\dist(z,\partial A)$, so $G\le H\le4G\le4\pi^{-1/2}$. Unbounded shape functionals, such as the normalized diameter, would require additional tail control.
\end{remark}

\begin{figure}[H]
\centering
\includegraphics[width=0.78\textwidth]{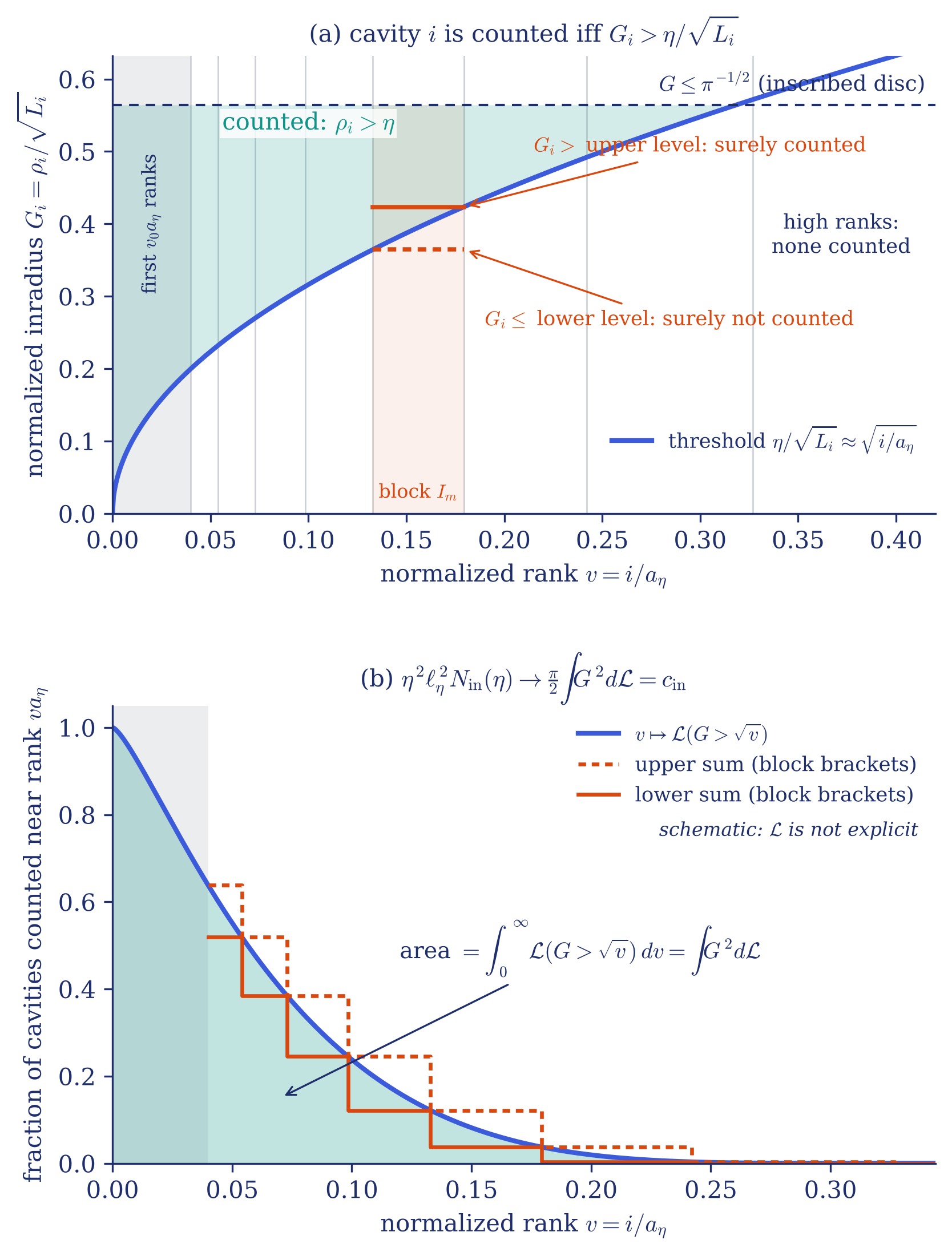}
\caption{\textbf{Moving threshold and layer cake.} (a) The cavity of rank $i$ is counted in $\Nin(\eta)$ if and only if $G_{i}>\eta/\sqrt{L_{i}}$, and $\eta/\sqrt{L_{i}}\approx\sqrt{i/a_{\eta}}$ by the rank--area law. On each geometric block of ranks the moving threshold lies between two fixed levels, so the count on the block is bracketed by two fixed-threshold counts, which Werner's theorem controls. Ranks beyond a multiple of $a_{\eta}$ are never counted since $G\le\pi^{-1/2}$, and the first $v_{0}a_{\eta}$ ranks contribute at most $2v_{0}a_{\eta}$. (b) Summing the brackets over the blocks gives upper and lower Riemann sums for $\int_{0}^{\infty}\cL(G>\sqrt{v})\,dv=\int G^{2}d\cL$. The multiplicative grid factors of the proof are omitted, and the curve is schematic, since $\cL$ is not explicit.}\label{fig:mechanism}
\end{figure}

\section{Moment control}\label{sec:moments}

The passage to moments rests on the domination of the inradius count by the area count and on the asymptotics of the second moment of the latter, derived from the window estimates underlying Le Gall's law.

\medskip
\noindent\textbf{Input III (window estimates).} \emph{Fix $\lambda\in(0,1)$ and, for $\varepsilon>0$, let}
\[
X_{\varepsilon}:=\sum_{i\ge1}L_{i}\,\mathbf{1}\bigl\{\pi\lambda^{2}\varepsilon^{2}<L_{i}\le\pi\varepsilon^{2}\bigr\}
=\Leb_{2}\bigl\{x:\pi\lambda^{2}\varepsilon^{2}<\Leb_{2}(C_{x})\le\pi\varepsilon^{2}\bigr\},
\]
\emph{where $C_{x}$ is the component of $\R^{2}\setminus K_{1}$ containing $x$. Then, as $\varepsilon\downarrow0$,}
\[
\E[X_{\varepsilon}]=\frac{\pi\log(1/\lambda)}{\ell_{\varepsilon}^{2}}+O\Bigl(\frac{\log\ell_{\varepsilon}}{\ell_{\varepsilon}^{5/2}}\Bigr),\qquad
\operatorname{Var}(X_{\varepsilon})=O\bigl(\ell_{\varepsilon}^{-11/2}\bigr).
\]
\emph{These are \cite[Propositions~18 and~20]{We94} for the constant functional $F\equiv1$; the set above is the set $\mathscr{U}_{\varepsilon}$ of \cite[\S4.2.1]{We94}, and Werner notes \cite[p.~323]{We94} that this case is the pair of estimates~3.1 and~3.2 of \cite{LG91}.} Since the standard deviation is of smaller order than the mean,
\begin{equation}\label{eq:l2win}
\norm{X_{\varepsilon}}_{2}=\bigl((\E X_{\varepsilon})^{2}+\operatorname{Var}(X_{\varepsilon})\bigr)^{1/2}=\frac{\pi\log(1/\lambda)}{\ell_{\varepsilon}^{2}}\,(1+o(1))\qquad(\varepsilon\downarrow0).
\end{equation}

\begin{proposition}[Second moments]\label{prop:second}
One has
\begin{equation}\label{eq:secondasym}
\limsup_{u\downarrow0}\E\bigl[(u\ell_{u}^{2}M(u))^{2}\bigr]\le\gamma^{2}.
\end{equation}
Consequently there are $u_{1}>0$ and $C<\infty$ with $\E[M(u)^{2}]\le Cu^{-2}\ell_{u}^{-4}$ for $0<u\le u_{1}$, and
\begin{equation}\label{eq:L2bounded}
\sup_{0<u\le1/2}\E\bigl[(u\ell_{u}^{2}M(u))^{2}\bigr]<\infty,\qquad
\sup_{0<\eta\le1/2}\E\bigl[(\eta^{2}\ell_{\eta}^{2}\Nin(\eta))^{2}\bigr]<\infty .
\end{equation}
\end{proposition}

\begin{proof}
Fix $\lambda\in(0,1)$ and $\delta>0$, and put $b:=\log(1/\lambda)$. By \eqref{eq:l2win} there is $\varepsilon_{\mathrm{w}}\in(0,e^{-1})$, depending on $\lambda$ and $\delta$, with $\norm{X_{\varepsilon}}_{2}\le(\pi b+\delta)\ell_{\varepsilon}^{-2}$ for $0<\varepsilon\le\varepsilon_{\mathrm{w}}$. Let $S:=\sum_{i}L_{i}$ be the total area of the cavities; they lie in the convex hull of $K_{1}$, so $S\le\pi R^{2}$ with $R:=\max_{t\le1}|B_{t}|$, and $\E[S^{2}]<\infty$. For $u\in(0,1)$ put $\varrho:=(u/(\pi\lambda^{2}))^{1/2}$, $\varrho_{k}:=\varrho\lambda^{-k}$ and $J_{k}:=(u\lambda^{-2k},u\lambda^{-2(k+1)}]$, $k\ge0$. These windows tile $(u,\infty)$, so $M(u)=\sum_{k}N_{k}$ with $N_{k}:=\card\{i:L_{i}\in J_{k}\}$, and since the cavities counted by $N_{k}$ have area larger than $u\lambda^{-2k}$,
\[
N_{k}\le\frac{\lambda^{2k}}{u}\,X_{\varrho_{k}} .
\]
Assume $\varrho\le\varepsilon_{\mathrm{w}}$ and let $k_{1}:=\max\{k:\varrho_{k}\le\varepsilon_{\mathrm{w}}\}$. The cavities counted by $N_{k}$, $k>k_{1}$, have area larger than $\pi\lambda^{2}\varepsilon_{\mathrm{w}}^{2}$, so $\sum_{k>k_{1}}N_{k}\le S/(\pi\lambda^{2}\varepsilon_{\mathrm{w}}^{2})$. Put $\chi:=\log(1/\varrho)=\frac12(\ell_{u}+\log(\pi\lambda^{2}))$ and $c_{\mathrm{w}}:=\log(1/\varepsilon_{\mathrm{w}})$, so that $\log(1/\varrho_{k})=\chi-kb\ge c_{\mathrm{w}}$ for $k\le k_{1}$. By the triangle inequality in $L^{2}$,
\[
u\ell_{u}^{2}\norm{M(u)}_{2}\le(\pi b+\delta)\sum_{k=0}^{k_{1}}\lambda^{2k}\frac{\ell_{u}^{2}}{(\chi-kb)^{2}}+\frac{u\ell_{u}^{2}\,\norm{S}_{2}}{\pi\lambda^{2}\varepsilon_{\mathrm{w}}^{2}} .
\]
The last term tends to $0$. For $k\le\chi/(2b)$, $\chi-kb\ge\chi/2$, and since $\chi=\frac12\ell_{u}+O(1)$, the $k$-th term is at most $64\lambda^{2k}$ for small $u$ and tends to $4\lambda^{2k}$ as $u\downarrow0$; by dominated convergence these terms contribute at most $4\sum_{k\ge0}\lambda^{2k}=4/(1-\lambda^{2})$ in the limit. For $\chi/(2b)<k\le k_{1}$, $\lambda^{2k}<e^{-\chi}=\varrho$, so these terms sum to at most $\varrho\,\ell_{u}^{2}(1-\lambda^{2})^{-1}c_{\mathrm{w}}^{-2}=O(u^{1/2}\ell_{u}^{2})\to0$. Hence
\[
\limsup_{u\downarrow0}u\ell_{u}^{2}\norm{M(u)}_{2}\le\frac{4(\pi b+\delta)}{1-\lambda^{2}} .
\]
Throughout, the window ratio $\lambda$ is fixed while $u\downarrow0$; it is sent to $1$ only after the small-area limit has been taken, so no uniformity of Input~III in $\lambda$ is needed. Let $\delta\downarrow0$, and then $\lambda\uparrow1$: since $4\pi\log(1/\lambda)/(1-\lambda^{2})\to2\pi=\gamma$, this proves \eqref{eq:secondasym}. The bound $\E[M(u)^{2}]\le Cu^{-2}\ell_{u}^{-4}$ for small $u$ follows, and for $u\in(u_{1},\frac12]$, $M(u)\le M(u_{1})$; this gives the first bound in \eqref{eq:L2bounded}.

For the second, if $\rho_{i}>\eta$ then $A_{i}$ contains a disc of radius larger than $\eta$ (Section~\ref{subsec:cavities}), so $L_{i}>\pi\eta^{2}$ and $\Nin(\eta)\le M(u)$ with $u:=\pi\eta^{2}/2$. For $\eta\le\frac12$, $\log(1/u)=2\ell_{\eta}-\log(\pi/2)\ge\ell_{\eta}$, whence
\begin{equation}\label{eq:domination}
\eta^{2}\ell_{\eta}^{2}\Nin(\eta)\le\frac{2}{\pi}\,u\ell_{u}^{2}M(u),\qquad u=\pi\eta^{2}/2 .\qedhere
\end{equation}
\end{proof}

\begin{proof}[Proof of Theorem~\ref{thm:A}]
The almost-sure law and the bounds on $\cin$ are Proposition~\ref{prop:as}. Let $Y_{u}:=u\ell_{u}^{2}M(u)$ and $Z_{\eta}:=\eta^{2}\ell_{\eta}^{2}\Nin(\eta)$, and fix sequences $u_{n}\downarrow0$, $\eta_{n}\downarrow0$. By Input~I, $Y_{u_{n}}\to\gamma$ almost surely, and by \eqref{eq:L2bounded} the $Y_{u_{n}}$ are bounded in $L^{2}$, hence uniformly integrable, so $\E[Y_{u_{n}}]\to\gamma$. By Fatou's lemma and \eqref{eq:secondasym}, $\E[Y_{u_{n}}^{2}]\to\gamma^{2}$. Therefore $\E[(Y_{u_{n}}-\gamma)^{2}]=\E[Y_{u_{n}}^{2}]-2\gamma\E[Y_{u_{n}}]+\gamma^{2}\to0$, and in particular the squares $Y_{u_{n}}^{2}$ are uniformly integrable. By \eqref{eq:domination}, $Z_{\eta_{n}}^{2}\le(4/\pi^{2})Y_{\pi\eta_{n}^{2}/2}^{2}$, so the squares $Z_{\eta_{n}}^{2}$ are uniformly integrable as well; since $Z_{\eta_{n}}\to\cin$ almost surely by Proposition~\ref{prop:as}, Vitali's theorem gives $Z_{\eta_{n}}\to\cin$ in $L^{2}$. Convergence in $L^{p}$ for $p\in[1,2]$ follows, and $\E[M(u)^{2}]\sim\gamma^{2}u^{-2}\ell_{u}^{-4}$.
\end{proof}

In particular $\operatorname{Var}(M(u))=o(u^{-2}\ell_{u}^{-4})$ and $\operatorname{Var}(\Nin(\eta))=o(\eta^{-4}\ell_{\eta}^{-4})$. No new probabilistic estimate is involved: the second-moment asymptotic is a deterministic $L^{2}$ summation of the window estimates over a geometric tiling of the area axis, in which the leading coefficient of Input~III is retained and the window ratio $\lambda$ is then sent to $1$.

\section{Lower bounds for the holes of the sausage}\label{sec:persistence}

Throughout this section $0<r_{0}<r_{1}<\infty$, $\psi\ge0$ is bounded and Borel with support in $[r_{0},r_{1}]$, and $I_{\psi}=\int_{r_{0}}^{r_{1}}\psi(r)r^{-2}\,dr$. All large-time statements are obtained from unit-time ones through Lemma~\ref{lem:transfer}, so the normalization is converted only once, by $(\log T)^{2}/T=4\varepsilon^{2}\ell_{\varepsilon}^{2}$ with $\varepsilon=T^{-1/2}$.

\begin{proposition}[Cavity window law at unit time]\label{prop:window}
\begin{itemize}
\item[(a)] Almost surely,
\[
\sup_{r\in[r_{0},r_{1}]}\bigl|\varepsilon^{2}\ell_{\varepsilon}^{2}\Nin(\varepsilon r)-\cin r^{-2}\bigr|\to0\qquad(\varepsilon\downarrow0);
\]
consequently $\varepsilon^{2}\ell_{\varepsilon}^{2}\int_{r_{0}}^{r_{1}}\Nin(\varepsilon r)\psi(r)\,dr\to\cin I_{\psi}$ almost surely.
\item[(b)] For every $r>0$, $\varepsilon^{2}\ell_{\varepsilon}^{2}\Nin(\varepsilon r)\to\cin r^{-2}$ in $L^{p}$ for every $p\in[1,2]$, and so does the integrated quantity in (a), to $\cin I_{\psi}$.
\end{itemize}
\end{proposition}

\begin{proof}
(a) On the event of Proposition~\ref{prop:as}, for fixed $r$ and $s=\varepsilon r$,
\[
\varepsilon^{2}\ell_{\varepsilon}^{2}\Nin(\varepsilon r)=r^{-2}(\ell_{\varepsilon}/\ell_{s})^{2}\,s^{2}\ell_{s}^{2}\Nin(s)\longrightarrow\cin r^{-2},
\]
since $\ell_{\varepsilon}/\ell_{\varepsilon r}\to1$ uniformly on $[r_{0},r_{1}]$. For uniformity, fix $\varkappa\in(0,1)$ and the points $t_{j}:=r_{0}(1+\varkappa)^{j}\wedge r_{1}$, $0\le j\le J$. Since $\Nin$ is nonincreasing, $\Nin(\varepsilon t_{j+1})\le\Nin(\varepsilon r)\le\Nin(\varepsilon t_{j})$ for $r\in[t_{j},t_{j+1}]$, and the normalized extremes converge, uniformly in the finitely many $j$, to $\cin t_{j+1}^{-2}\ge\cin(1+\varkappa)^{-2}r^{-2}$ and $\cin t_{j}^{-2}\le\cin(1+\varkappa)^{2}r^{-2}$. Hence the $\limsup$ of the supremum is at most $\cin r_{0}^{-2}((1+\varkappa)^{2}-1)$; let $\varkappa\downarrow0$. Since $r\mapsto\Nin(\varepsilon r)$ is right-continuous and $r\mapsto\cin r^{-2}$ is continuous, the supremum over $[r_{0},r_{1}]$ equals the supremum over the countable set $(\mathbb{Q}\cap[r_{0},r_{1}])\cup\{r_{0},r_{1}\}$; the right endpoint must be included, since it is not a limit from the right of points of the interval. By Lemma~\ref{lem:measholes}, the supremum is therefore measurable. The integrated statement follows, with error at most $\norm{\psi}_{\infty}(r_{1}-r_{0})$ times the uniform error.

(b) For fixed $r$ this is Theorem~\ref{thm:A}, after the change of scale $s=\varepsilon r$. For the integrated quantity, monotonicity of $\Nin$ gives the domination
\[
\varepsilon^{2}\ell_{\varepsilon}^{2}\int_{r_{0}}^{r_{1}}\Nin(\varepsilon r)\psi(r)\,dr\le\norm{\psi}_{\infty}(r_{1}-r_{0})\,r_{0}^{-2}\bigl(\ell_{\varepsilon}/\ell_{\varepsilon r_{0}}\bigr)^{2}\,(\varepsilon r_{0})^{2}\ell_{\varepsilon r_{0}}^{2}\Nin(\varepsilon r_{0}),
\]
whose square is uniformly integrable as $\varepsilon\downarrow0$, by the convergence in $L^{2}$ in Theorem~\ref{thm:A}; with (a), this gives convergence in $L^{p}$ for $p\in[1,2]$.
\end{proof}

\begin{proof}[Proof of Theorem~\ref{thm:B}]
By Lemma~\ref{lem:comparison}, pathwise and for all $\varepsilon,r>0$,
\begin{equation}\label{eq:pathwise}
\beta_{1}\bigl(K_{1}^{(\varepsilon r)}\bigr)\ge\Nin(\varepsilon r),\qquad\beta_{1}^{T}(r)\ge N_{T}(r):=N_{K_{T}}(r).
\end{equation}
Part (a) follows from \eqref{eq:pathwise} and Proposition~\ref{prop:window}(a), taking $r=1$ for the first inequality. For the final assertion, put $\Psi_{\psi}(T):=\int_{r_{0}}^{r_{1}}N_{T}(r)\psi(r)\,dr$, a random variable by Lemma~\ref{lem:measholes}. By Lemma~\ref{lem:transfer}, $(\log T)^{2}T^{-1}N_{T}(r)$ and $(\log T)^{2}T^{-1}\Psi_{\psi}(T)$ have the laws of $4\varepsilon^{2}\ell_{\varepsilon}^{2}\Nin(\varepsilon r)$ and $4\varepsilon^{2}\ell_{\varepsilon}^{2}\int\Nin(\varepsilon r)\psi(r)\,dr$ with $\varepsilon=T^{-1/2}$, so by Proposition~\ref{prop:window} they converge to $4\cin r^{-2}$ and $4\cin I_{\psi}$ in probability and in $L^{p}$, $p\in[1,2]$. By \eqref{eq:pathwise} and \eqref{eq:Phi}, $\beta_{1}^{T}(r)\ge N_{T}(r)$ and $\Phi_{\psi}(T)\ge\Psi_{\psi}(T)$. Convergence in $L^{1}$ gives the expectation bounds in (b) and (c). Since $\cin>0$, convergence in probability gives $\Prob(N_{T}(r)\ge(1-\theta)4\cin r^{-2}T(\log T)^{-2})\to1$, hence the probability bound in (b), and likewise in (c) when $I_{\psi}>0$; when $I_{\psi}=0$ the event in (c) is certain, since $\Phi_{\psi}(T)\ge0$.
\end{proof}

Part (a) of Theorem~\ref{thm:B} holds on a single event of probability one, for all small $\varepsilon$ at once; parts (b) and (c) are not almost-sure statements along the original path, because Brownian scaling is an equality in law at each fixed $T$. If $\E[\beta_{1}^{T}(r)]\sim\kappa r^{-2}T/(\log T)^{2}$, then $\kappa\ge4\cin$ by Theorem~\ref{thm:B}(b), and equivalently $\kappa_{0}:=\lim_{\varepsilon\downarrow0}\varepsilon^{2}\ell_{\varepsilon}^{2}\E[\beta_{1}(K_{1}^{(\varepsilon)})]=\kappa/4\ge\cin$. The value of $\cin$ is a moment of Werner's law and is not known in closed form; Werner obtains $\cL$ from an explicitly described limiting object related to the Brownian loop \cite{We94}, which may make a numerical evaluation easier than a simulation of the trace. As for Le Gall's law, the convergence is logarithmic and should not be expected to be visible at moderate scales.

\section{The upper bound: near-critical pockets}\label{sec:upper}

\subsection{No fixed-radius bound in terms of coarse quantities}\label{subsec:nodet}

\begin{remark}[No fixed-radius bound in terms of coarse quantities]\label{rem:nodet}
By \cite[Theorem~1.2 and Corollary~6.3]{GuPS}, for every $\varepsilon\in(0,1/10]$ there is a connected rectifiable curve $J_{\varepsilon}$, contained in a fixed disc of radius $4$ and of length at most $12$, whose unit parallel set has at least $1/(2\varepsilon)$ holes, while the area and the boundary length of that parallel set stay bounded along the family. Hence no inequality $\beta_{1}(A^{(r)})\le F(\diam A,\operatorname{length}A,|A^{(r)}|,\mathcal{H}^{1}(\partial A^{(r)}))$ with $F$ bounded on bounded sets holds for all connected rectifiable curves $A$ at a prescribed radius; by dilation, there is no constant $C$ with $\beta_{1}(A^{(r)})\le C\mathcal{H}^{1}(\partial A^{(r)})/r$ for all compact connected $A$. The mechanism is that two almost tangent pieces of a parallel set enclose tiny holes of almost vanishing perimeter. Deterministic estimates using finer information, such as Lemma~\ref{lem:recursion} below, remain possible; but an upper bound of the correct order for the Wiener sausage requires probabilistic input: such configurations must be shown to be rare for the Brownian trace, which is done in \cite{GuII}; Figure~\ref{fig:tangency} illustrates the mechanism.
\end{remark}

\begin{figure}[t]
\centering
\includegraphics[width=\textwidth]{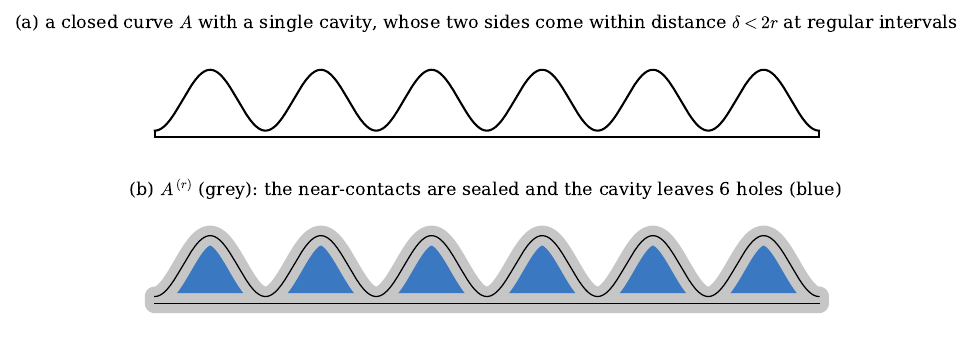}
\caption{The mechanism behind Remark~\ref{rem:nodet}, schematically. (a) A closed curve $A$ with a single cavity, whose two sides come within distance $\delta<2r$ of each other at regular intervals. (b) In $A^{(r)}$ the near-contacts are sealed, and the cavity leaves one hole between consecutive near-contacts. In the construction of \cite{GuPS}, near-contacts repeated at smaller and smaller scales give connected curves of bounded length whose parallel sets at a fixed radius have arbitrarily many holes, with bounded area and perimeter.}\label{fig:tangency}
\end{figure}

The integrated bound \eqref{eq:integrated} shows that the radius integration removes this obstruction, but only at the polynomial order $T^{2}$. In contrast, hole counts \emph{uniform} in the radius are controlled by the number of components of local maxima of $d_{A}$ \cite[Theorem~1.3]{GuPS}; the near-critical pockets below are a dyadic form of that quantity.

\subsection{Near-critical pockets and the dyadic depth recursion}\label{subsec:nc}

Let $A\subset\R^{2}$ be compact and $\eta>0$. If $U$ is a hole of $A^{(\eta)}$, that is, a bounded component of $\{d_{A}>\eta\}$, then $d_{A}=\eta$ on $\partial U$, so the maximum of $d_{A}$ over $\overline U$, the \emph{depth} of $U$, exceeds $\eta$ and is attained in $U$.

\begin{definition}[Near-critical pockets]\label{def:nc}
A hole $U$ of $A^{(\eta)}$ is a \emph{near-critical pocket} at level $\eta$ if its depth is at most $2\eta$. We write $\Ncrit_{A}(\eta)$ for their number, and $\Ncrit(\eta):=\Ncrit_{K_{1}}(\eta)$.
\end{definition}

By Lemma~\ref{lem:measholes}, $\Ncrit(\eta)$ is a random variable, and trivially $\Ncrit_{A}(\eta)\le\beta_{1}(A^{(\eta)})$. A near-critical pocket at level $\eta$ contains no disc of radius larger than $\eta$: if $B(x,\rho)\subset U$, moving from $x$ towards a nearest point of $A$ shows $d_{A}(x)\ge\eta+\rho$. As $r$ increases from $\eta$, it may split into several holes of $A^{(r)}$, all of which have disappeared by $r=2\eta$. A point where its depth is attained is a local maximum of $d_{A}$, and its component in the set of local maxima of $d_{A}$, on which $d_{A}$ is constant \cite[Lemma~5.1]{GuPS}, lies in $U$.

\begin{lemma}[Dyadic depth recursion]\label{lem:recursion}
Let $A\subset\R^{2}$ be compact with diameter $D$, and let $\varepsilon=\lambda_{0}<\lambda_{1}<\dots<\lambda_{m}$ with $\lambda_{j+1}\le2\lambda_{j}$ and $\lambda_{m}\ge2\varepsilon$. Then
\[
\beta_{1}\bigl(A^{(\varepsilon)}\bigr)\le\sum_{j=0}^{m-1}\Ncrit_{A}(\lambda_{j})+\frac{4(D+\varepsilon)^{2}}{\lambda_{m}^{2}} .
\]
\end{lemma}

\begin{proof}
Let $V$ be a hole of $A^{(\varepsilon)}$ and $x^{*}\in V$ a point where its depth is attained; let $j$ be the largest index with $\lambda_{j}<d_{A}(x^{*})$. If $j<m$, the component $V_{j}$ of $\{d_{A}>\lambda_{j}\}$ containing $x^{*}$ is a connected subset of $\{d_{A}>\varepsilon\}$ meeting $V$, hence $V_{j}\subset V$, and its depth is at most $d_{A}(x^{*})\le\lambda_{j+1}\le2\lambda_{j}$: it is a near-critical pocket at level $\lambda_{j}$. The assignment $V\mapsto(j,V_{j})$ is injective, since distinct holes $V$ are disjoint. If $j=m$, then $d_{A}(x^{*})>\lambda_{m}$, and the disc $B(x^{*},d_{A}(x^{*})-\varepsilon)$ lies in $\{d_{A}>\varepsilon\}$, hence in $V$, with radius larger than $\lambda_{m}-\varepsilon\ge\lambda_{m}/2$. These discs are disjoint and lie in $\overline{B}(a,D+\varepsilon)$ for any $a\in A$, because $\{|y-a|>D+\varepsilon\}$ is a connected unbounded subset of $\{d_{A}>\varepsilon\}$. Comparing areas bounds their number by $4(D+\varepsilon)^{2}/\lambda_{m}^{2}$.
\end{proof}

\begin{figure}[t]
\centering
\includegraphics[width=0.95\textwidth]{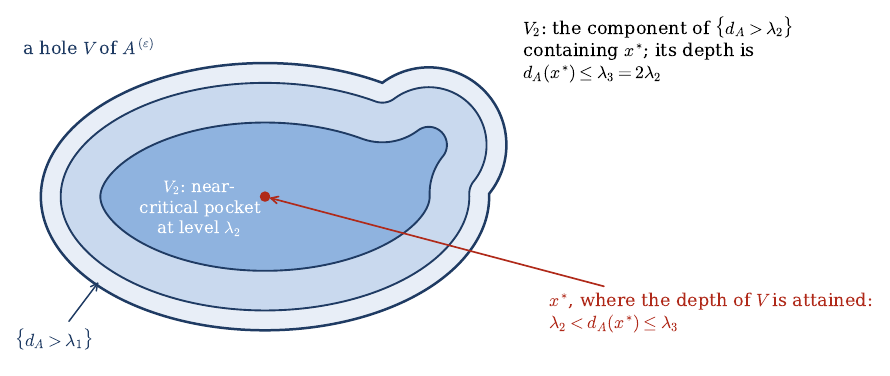}
\caption{The dyadic depth recursion (Lemma~\ref{lem:recursion}) on an example, with $\lambda_{j}=2^{j}\varepsilon$. The depth of the hole $V$ of $A^{(\varepsilon)}$ is attained at $x^{*}$ and lies in $(\lambda_{2},\lambda_{3}]$. The component $V_{2}$ of $\{d_{A}>\lambda_{2}\}$ containing $x^{*}$ lies in $V$ and has depth at most $\lambda_{3}=2\lambda_{2}$, so it is a near-critical pocket at level $\lambda_{2}$, to which $V$ is charged. The nested regions are drawn as if $d_{A}$ were $\varepsilon$ plus the distance to $\partial V$.}\label{fig:recursion}
\end{figure}

The lemma (Figure~\ref{fig:recursion}) charges every hole, including the ephemeral ones responsible for Remark~\ref{rem:nodet}, to a near-critical pocket at a single scale, and uses packing only for deep holes, where it is legitimate.

\subsection{Decomposition by complementary components}\label{subsec:decomposition}

Let $A_{\infty}$ be the unbounded component of $\R^{2}\setminus K_{1}$. For $\varepsilon>0$ define the inward erosions
\[
A_{i}^{-\varepsilon}:=\{x\in A_{i}:d_{K_{1}}(x)>\varepsilon\},\qquad A_{\infty}^{-\varepsilon}:=\{x\in A_{\infty}:d_{K_{1}}(x)>\varepsilon\},
\]
written with a minus sign to distinguish them from offsets; let $m_{i}(\varepsilon)$ be the number of components of $A_{i}^{-\varepsilon}$, and $m_{\infty}(\varepsilon)$ the number of \emph{bounded} components of $A_{\infty}^{-\varepsilon}$. Thus $m_{i}(\varepsilon)\ge1$ if and only if $\rho_{i}>\varepsilon$. The set $A_{\infty}^{-\varepsilon}$ has exactly one unbounded component, the one containing $\{|x|\ge R\}$ for $R$ large.

\begin{proposition}[Three-term identity]\label{prop:threeterm}
For every $\varepsilon>0$, in $\{0,1,\dots\}\cup\{\infty\}$,
\[
\beta_{1}\bigl(K_{1}^{(\varepsilon)}\bigr)=\sum_{i\ge1}m_{i}(\varepsilon)+m_{\infty}(\varepsilon)=\Nin(\varepsilon)+N_{\mathrm{split}}(\varepsilon)+m_{\infty}(\varepsilon),
\]
where $N_{\mathrm{split}}(\varepsilon):=\sum_{i\ge1}(m_{i}(\varepsilon)-1)_{+}$.
\end{proposition}

\begin{proof}
A hole of $K_{1}^{(\varepsilon)}$ is a connected subset of $\R^{2}\setminus K_{1}$, hence lies in exactly one $A_{i}$, where it is a component of $A_{i}^{-\varepsilon}$, or in $A_{\infty}$, where it is a bounded component of $A_{\infty}^{-\varepsilon}$. Conversely, a component of $A_{i}^{-\varepsilon}$ is a connected subset of $\{d_{K_{1}}>\varepsilon\}$, and any connected subset of $\{d_{K_{1}}>\varepsilon\}$ containing it lies in $A_{i}$; so it is a component of $\{d_{K_{1}}>\varepsilon\}$, bounded since $A_{i}$ is. The same holds for bounded components of $A_{\infty}^{-\varepsilon}$. The second form follows from $m_{i}(\varepsilon)=\mathbf{1}\{\rho_{i}>\varepsilon\}+(m_{i}(\varepsilon)-1)_{+}$.
\end{proof}

All terms are finite almost surely for each fixed $\varepsilon$, since $\E[\beta_{1}(K_{1}^{(\varepsilon)})]<\infty$. The outer term may be positive: a nearly closed circular arc has no cavity but its parallel set has a hole (Fig.~\ref{fig:outer}), and inward erosion can seal a fjord of $A_{\infty}$ when two arcs of the frontier of the trace come within $2\varepsilon$ of each other. The first term is governed by Theorem~\ref{thm:A}; the other two are $O(\varepsilon^{-2}\ell_{\varepsilon}^{-2})$ in expectation by Corollary~\ref{cor:order}. The outer term is of lower order, $\E[m_{\infty}(\varepsilon)]=O(\varepsilon^{-2+\theta})$ for an absolute $\theta>0$ \cite[Theorem~1.3]{GuII}, and the splitting term is of exact order $\varepsilon^{-2}\ell_{\varepsilon}^{-2}$ in expectation \cite[Theorem~1.4]{GuII}; the existence of its normalized limit is open (Section~\ref{sec:outlook}).

\begin{figure}[H]
\centering
\includegraphics[width=\textwidth]{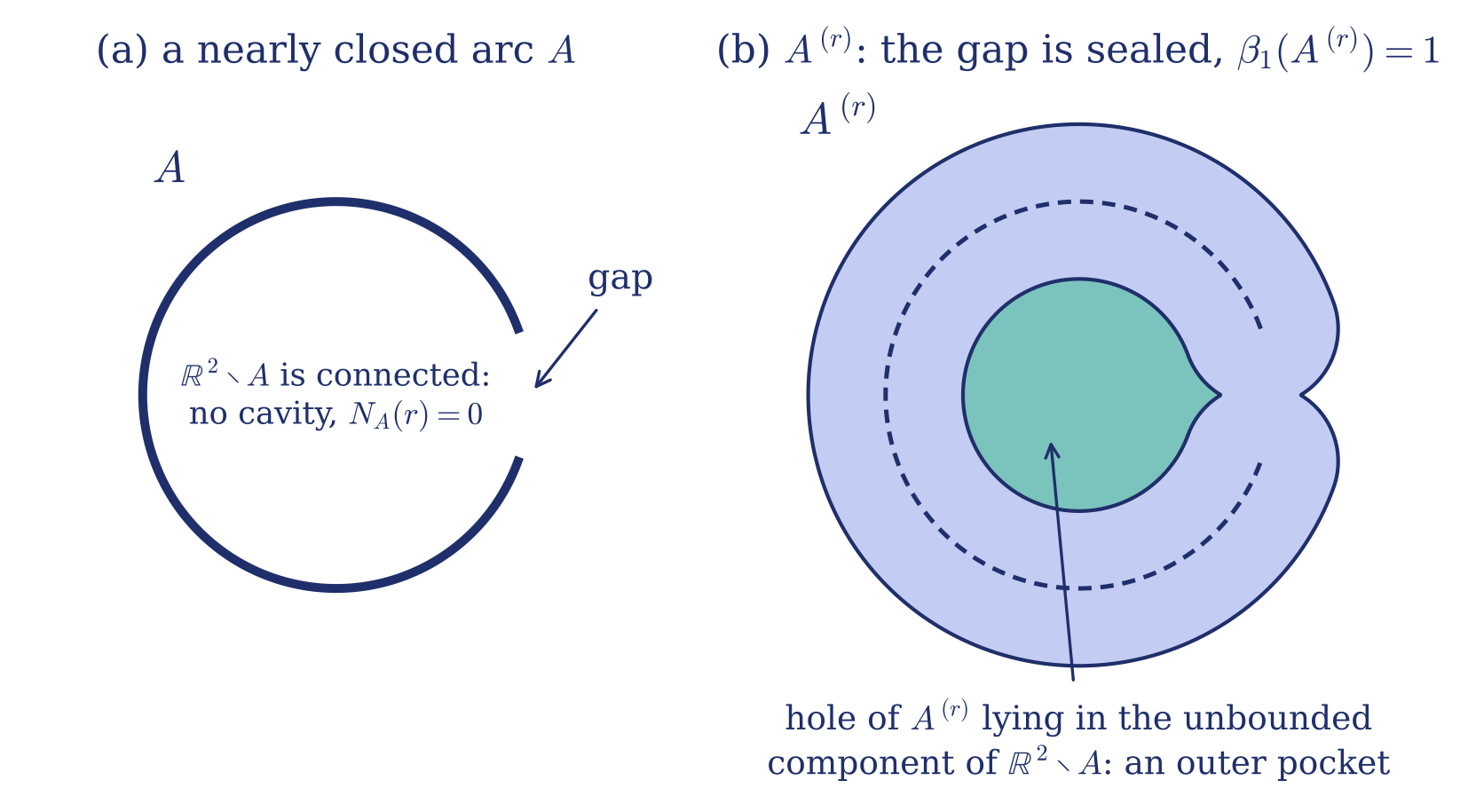}
\caption{\textbf{An outer pocket.} (a) A nearly closed circular arc $A$ has no cavity: $\R^{2}\setminus A$ is connected. (b) As soon as $r$ exceeds half the distance between the endpoints of the arc, and as long as $r$ is smaller than the distance from the centre to $A$, the offset $A^{(r)}$ has a hole, which lies in the unbounded component of $\R^{2}\setminus A$. Such holes are counted by the outer term $m_{\infty}$ of Proposition~\ref{prop:threeterm}; for the Brownian trace they arise when two arcs of its frontier come within $2\varepsilon$ of each other.}\label{fig:outer}
\end{figure}

The same classification applies to near-critical pockets. A cavity with $\rho_{i}\le\lambda$ has no pocket at level $\lambda$. A cavity with $\rho_{i}\in(\lambda,2\lambda]$ has $m_{i}(\lambda)\ge1$ pockets, all near-critical, since their depths are at most $\rho_{i}$. A cavity with $\rho_{i}>2\lambda$ has a pocket of depth larger than $2\lambda$, hence at most $m_{i}(\lambda)-1$ near-critical ones. Therefore
\begin{equation}\label{eq:ncsplit}
\Ncrit(\lambda)=\bigl[\Nin(\lambda)-\Nin(2\lambda)\bigr]+\Nres(\lambda),
\end{equation}
\[
0\le\Nres(\lambda)\le N_{\mathrm{split}}(\lambda)+m_{\infty}(\lambda),
\]
where $\Nin(\lambda)-\Nin(2\lambda)=\card\{i:\rho_{i}\in(\lambda,2\lambda]\}$ and $\Nres(\lambda)$, defined in the introduction, counts the near-critical pockets in cavities other than one per cavity with $\rho_{i}\in(\lambda,2\lambda]$, together with the near-critical bounded components of $A_{\infty}^{-\lambda}$; which pocket of a cavity with $\rho_{i}\in(\lambda,2\lambda]$ is regarded as compulsory does not affect the count.

\begin{lemma}[Telescoping]\label{lem:telescoping}
For $\varepsilon>0$ and $m\ge1$, with $D=\diam K_{1}$,
\[
\beta_{1}\bigl(K_{1}^{(\varepsilon)}\bigr)\le\Nin(\varepsilon)+\sum_{j=0}^{m-1}\Nres\bigl(2^{j}\varepsilon\bigr)+\frac{4(D+\varepsilon)^{2}}{4^{m}\varepsilon^{2}} .
\]
\end{lemma}

\begin{proof}
Apply Lemma~\ref{lem:recursion} with $\lambda_{j}=2^{j}\varepsilon$ and insert \eqref{eq:ncsplit}; the cavity terms telescope to $\Nin(\varepsilon)-\Nin(2^{m}\varepsilon)\le\Nin(\varepsilon)$.
\end{proof}

\subsection{Proof of Theorem~\ref{thm:C}}\label{subsec:proofC}

\emph{(i)$\Rightarrow$(ii)$\Rightarrow$(iii).} The first implication follows from $\Ncrit\le\beta_{1}(K_{1}^{(\cdot)})$, and the second from $\Nres\le\Ncrit$, by \eqref{eq:ncsplit}.

\emph{(iii)$\Rightarrow$(i), with the constant.} Assume (iii) with constants $C_{\circ},\eta_{0}$, where we may take $\eta_{0}\le1/2$. Let $0<\varepsilon\le\eta_{0}$, write $\ell:=\ell_{\varepsilon}$, and let $m\ge1$ be the smallest integer with $2^{m}\varepsilon>\eta_{0}$, so that $\lambda_{j}:=2^{j}\varepsilon\le\eta_{0}$ for $j<m$ and $4^{m}\varepsilon^{2}>\eta_{0}^{2}$. Taking expectations in Lemma~\ref{lem:telescoping},
\[
\E\bigl[\beta_{1}\bigl(K_{1}^{(\varepsilon)}\bigr)\bigr]\le\E[\Nin(\varepsilon)]+C_{\circ}\sum_{j=0}^{m-1}\frac{1}{\lambda_{j}^{2}\ell_{\lambda_{j}}^{2}}+\frac{4\,\E[(D+1)^{2}]}{\eta_{0}^{2}} .
\]
The last term is a constant, since $\diam K_{1}\le2\max_{t\le1}|B_{t}|$ has moments of all orders. For $\lambda_{j}\le\sqrt\varepsilon$ one has $\ell_{\lambda_{j}}=\ell-j\log2\ge\ell/2$, so
\[
\varepsilon^{2}\ell^{2}\,\frac{1}{\lambda_{j}^{2}\ell_{\lambda_{j}}^{2}}=4^{-j}\Bigl(\frac{\ell}{\ell-j\log2}\Bigr)^{2}\le4\cdot4^{-j},
\]
and each term tends to $4^{-j}$ as $\varepsilon\downarrow0$; by dominated convergence these terms contribute at most $\sum_{j\ge0}4^{-j}=4/3$ in the limit. For small $\varepsilon$ there are at most $\log_{2}(\eta_{0}/\sqrt\varepsilon)+1\le\ell$ levels with $\sqrt\varepsilon<\lambda_{j}\le\eta_{0}$, each with $\lambda_{j}^{-2}<\varepsilon^{-1}$ and $\ell_{\lambda_{j}}\ge\log2$, so they contribute $O(\ell\varepsilon^{-1})=o(\varepsilon^{-2}\ell^{-2})$. Since $\varepsilon^{2}\ell^{2}\E[\Nin(\varepsilon)]\to\cin$ by Theorem~\ref{thm:A},
\begin{equation}\label{eq:limsupconst}
\limsup_{\varepsilon\downarrow0}\varepsilon^{2}\ell_{\varepsilon}^{2}\,\E\bigl[\beta_{1}\bigl(K_{1}^{(\varepsilon)}\bigr)\bigr]\le\cin+\tfrac43C_{\circ} ,
\end{equation}
which gives (i) with any $C_{+}>\cin+\frac43C_{\circ}$ and a suitable $\varepsilon_{0}$.

\emph{(ii)$\Rightarrow$(i) without Theorem~\ref{thm:A}.} If (ii) holds with constants $C',\eta_{0}$, Lemma~\ref{lem:recursion} with $\lambda_{j}=2^{j}\varepsilon$ and the same choice of $m$ give $\E[\beta_{1}(K_{1}^{(\varepsilon)})]\le C'\sum_{j<m}\lambda_{j}^{-2}\ell_{\lambda_{j}}^{-2}+4\E[(D+1)^{2}]/\eta_{0}^{2}$, and the computation above yields $\limsup_{\varepsilon\downarrow0}\varepsilon^{2}\ell_{\varepsilon}^{2}\E[\beta_{1}(K_{1}^{(\varepsilon)})]\le\frac43C'$. This uses only the moments of $\diam K_{1}$.

\emph{Fixed radius.} For $r>0$ and $\varepsilon_{T}:=r/\sqrt{T}$, Lemma~\ref{lem:transfer} gives $\E[\beta_{1}^{T}(r)]=\E[\beta_{1}(K_{1}^{(\varepsilon_{T})})]$, and $\varepsilon_{T}^{-2}\ell_{\varepsilon_{T}}^{-2}=T r^{-2}(\frac12\log T-\log r)^{-2}\sim4Tr^{-2}(\log T)^{-2}$; so \eqref{eq:limsupconst} gives
\[
\limsup_{T\to\infty}\frac{(\log T)^{2}}{T}\,\E\bigl[\beta_{1}^{T}(r)\bigr]\le\frac{c_{+}}{r^{2}}.
\]

\emph{Smoothed.} Let $f_{T}(r):=\psi(r)(\log T)^{2}T^{-1}\E[\beta_{1}^{T}(r)]$, so that $(\log T)^{2}T^{-1}\E[\Phi_{\psi}(T)]=\int f_{T}$ by \eqref{eq:Phi} and Tonelli's theorem. Fix $C_{+}$ and $\varepsilon_{0}$ as in (i), and $T_{0}$ with $r_{1}/\sqrt{T_{0}}\le\varepsilon_{0}$ and $\frac16\log T_{0}\ge\max(\log r_{1},0)$. For $T\ge T_{0}$ and $r\in[r_{0},r_{1}]$, $\ell_{\varepsilon_{T}}=\frac12\log T-\log r\ge\frac13\log T$, so $0\le f_{T}(r)\le9C_{+}\norm{\psi}_{\infty}r_{0}^{-2}$. By the reverse Fatou lemma and the fixed-radius bound, $\limsup_{T}\int f_{T}\le\int\limsup_{T}f_{T}\le c_{+}I_{\psi}$. \qed

\medskip
By \cite[Theorem~1.2]{GuII}, Hypotheses (NC) and (NC$_{\circ}$) hold, so the upper bounds of Theorem~\ref{thm:C} are unconditional and have the same order as the lower bounds of Theorem~\ref{thm:B}. If (NC) holds with the constant $C$, the implication (ii)$\Rightarrow$(i) above gives $\limsup_{T}(\log T)^{2}T^{-1}\E[\beta_{1}^{T}(r)]\le16C/(3r^{2})$. If (NC$_{\circ}$) holds with the constant $C_{\circ}$ and a sharp constant $\kappa$ exists, then $4\cin\le\kappa\le4\cin+\frac{16}{3}C_{\circ}$, and $C_{\circ}$ may be replaced by $\limsup_{\eta\downarrow0}\eta^{2}\ell_{\eta}^{2}\E[\Nres(\eta)]$. If such a constant exists, the smoothed law becomes exact.

\begin{corollary}[Exact-constant sharpening]\label{cor:exact}
Suppose that $\E[\beta_{1}^{T}(r)]\sim\kappa r^{-2}T/(\log T)^{2}$ for every $r>0$, with $\kappa>0$. Then for every bounded Borel $\psi$ supported in $[r_{0},r_{1}]$, nonnegative or signed, $(\log T)^{2}T^{-1}\E[\Phi_{\psi}(T)]\to\kappa I_{\psi}$; in particular $\E[\Phi_{\psi}(T)]\sim\kappa I_{\psi}T/(\log T)^{2}$ whenever $I_{\psi}\ne0$.
\end{corollary}

\begin{proof}
By \eqref{eq:Phi} and Fubini's theorem, which applies since $\E[\Phi_{|\psi|}(T)]<\infty$, the normalized expectation is $\int\psi(r)(\log T)^{2}T^{-1}\E[\beta_{1}^{T}(r)]\,dr$, with integrand converging to $\kappa\psi(r)r^{-2}$. The hypothesis at $r=1$ gives $S_{0}$ with $\E[\beta_{1}^{S}(1)]\le(\kappa+1)S/(\log S)^{2}$ for $S\ge S_{0}$, and $\E[\beta_{1}^{T}(r)]=\E[\beta_{1}^{T/r^{2}}(1)]$ by Lemma~\ref{lem:transfer}. For $T$ so large that $T/r_{1}^{2}\ge S_{0}$ and $\log(T/r^{2})\ge\frac12\log T$ on $[r_{0},r_{1}]$, the integrand is bounded by $4(\kappa+1)r_{0}^{-2}|\psi(r)|$, and dominated convergence applies.
\end{proof}

For a signed weight $\psi=\psi_{+}-\psi_{-}$, Theorems~\ref{thm:B} and~\ref{thm:C} give, since (NC$_{\circ}$) holds, the bounds $4\cin I_{+}-c_{+}I_{-}\le\liminf\le\limsup\le c_{+}I_{+}-4\cin I_{-}$ for $(\log T)^{2}T^{-1}\E[\Phi_{\psi}(T)]$, where $I_{\pm}:=\int\psi_{\pm}(r)r^{-2}dr$; a single coefficient in front of $I_{\psi}$ requires the exact constant. In the exact case, $\Lambda(\phi_{\psi}):=\kappa I_{\psi}$ defines a mean local intensity of the persistence diagram on the Betti-curve test class (well defined, since $\phi_{\psi}$ determines $\psi$ almost everywhere); this is not a weak-limit theorem for diagram measures.

\section{The sharp constant and further questions}\label{sec:outlook}

\emph{Hypothesis (NC).} The proof of (NC) in \cite{GuII} follows the structure suggested by the pinning of near-critical pockets: at a point where the depth of a pocket is attained, the directions to the nearest points of the trace meet every closed half-circle \cite[Lemma~3.1]{GuPS}, and a robust form of this fact, for the points of the trace in a thin annulus seen from a nearby grid point \cite[Lemma~2.2]{GuII}, shows that the pocket is created by a constrained contact configuration of the trace at scale $\eta$. Each pocket is assigned to one of $\lceil\eta^{-2}\rceil$ deterministic time intervals, the contact configurations are priced along the two independent arms of the path from a deterministic root by exterior harmonic measure, and each arm pays one factor $1/\ell_{\eta}$ for avoiding the pocket after its last contact. The constant obtained there is not identified, and the cavity part $\cin$ of the lower bound is not recovered by that argument.

\emph{The sharp constant.} Write $N_{\mathrm{split}}$ and $m_{\infty}$ for the splitting and outer terms of Proposition~\ref{prop:threeterm}. Two facts are known from \cite{GuII}. The outer term is of lower order: $\E[m_{\infty}(\varepsilon)]\le C\varepsilon^{-2+\theta}$ for an absolute $\theta>0$ \cite[Theorem~1.3]{GuII}. This is in line with the heuristic that a fjord sealed at scale $\varepsilon$ requires two arcs of the frontier of the trace, a set of dimension $4/3$ \cite{LSW01}, to come within $2\varepsilon$ while the enclosed region stays connected to infinity in the raw trace, a non-disconnection constraint of the type measured by Lawler's exponents \cite{La96}; neither the dimension of the frontier nor the topology alone proves it, and the proof in \cite{GuII} combines the first-moment estimates of that paper, used stratum by stratum in the margin of the pockets, with a packing argument near the frontier and an elementary one-arm disconnection bound, with a non-explicit exponent $\theta$. The splitting term is of leading order: $\liminf_{\varepsilon\downarrow0}\varepsilon^{2}\ell_{\varepsilon}^{2}\E[N_{\mathrm{split}}(\varepsilon)]\ge c_{*}$ for an absolute constant $c_{*}>0$ \cite[Theorem~1.4]{GuII}, proved by forcing, at deterministic times, dumbbell-shaped cavities that erosion cuts in two. Consequently, if
\[
c_{\mathrm{split}}:=\lim_{\varepsilon\downarrow0}\varepsilon^{2}\ell_{\varepsilon}^{2}\E[N_{\mathrm{split}}(\varepsilon)]
\]
exists, Proposition~\ref{prop:threeterm} and Theorem~\ref{thm:A} give $\kappa_{0}=\cin+c_{\mathrm{split}}$ with $c_{\mathrm{split}}\ge c_{*}$, and $\kappa=4\kappa_{0}$; and any limit $\kappa$ of $(\log T)^{2}T^{-1}\E[\beta_{1}^{T}(r)]r^{2}$ satisfies $\kappa\ge4(\cin+c_{*})$. The existence of $c_{\mathrm{split}}$ is the remaining question. For a candidate value, write $m(\sigma,t)$ for the number of components of $\{x\in\sigma:\dist(x,\partial\sigma)>t\}$, for $\sigma$ of unit area; then $m_{i}(\varepsilon)=m(\sigma_{i},\varepsilon/\sqrt{L_{i}})$ with $\sigma_{i}:=L_{i}^{-1/2}A_{i}$. If one assumes, formally, that the area--shape marks $(L_{i},\sigma_{i})$ have at small areas the product intensity $\gamma\,du/(u^{2}\ell_{u}^{2})\otimes\cL(d\sigma)$, the computation behind Theorem~\ref{thm:A} gives
\[
c_{\mathrm{split}}=\frac{\gamma}{4}\int\cL(d\sigma)\int_{0}^{G(\sigma)}2t\,\bigl(m(\sigma,t)-1\bigr)_{+}\,dt ,
\]
a moment of Werner's law, consistent with $\cin=\frac{\gamma}{4}\int\cL(d\sigma)\int_{0}^{\infty}2t\,\mathbf{1}\{t<G(\sigma)\}\,dt$. This is a heuristic: the product-intensity assumption is not supplied by Werner's theorem, which concerns bounded test functions along the area ranking, while $m$ is unbounded and need not be continuous in $t$, and finiteness requires a bound such as $\int m(\sigma,t)\,\cL(d\sigma)=O(t^{-2+\delta})$, which cannot follow from bounds in terms of coarse geometric quantities (Remark~\ref{rem:nodet}). Whether the limit $c_{\mathrm{split}}$ exists, and whether this formula gives its value, are open.

\emph{Further questions.} Theorem~\ref{thm:B} gives the lower bound in probability, and Corollary~\ref{cor:order} shows that $(\log T)^{2}T^{-1}\beta_{1}^{T}(r)$ is bounded in probability as $T\to\infty$; the existence of the limit of its expectation, and a law of large numbers, are open. In the recurrent case there is no regeneration structure to provide them. A natural route to a law of large numbers is through estimates for pairs of near-critical pockets, including pairs sharing contacts. The Euler identity of Remark~\ref{rem:euler}, combined with distributional Gauss--Bonnet formulas for parallel sets \cite{Fe59,Fu85,HLW04} and mean curvature functions of Brownian paths \cite{La06}, offers a curvature route to the sharp constant, in which the splitting and outer terms reappear as defects at critical radii; by Remark~\ref{rem:nodet}, these defects cannot be controlled by coarse geometric quantities. Finally, it is natural to ask for test classes larger than the Betti-curve class, and for a local description of the persistence diagram near the origin of the birth--death plane.


\begin{thebibliography}{99}

\bibitem{Ba25}
Y. Baryshnikov, \emph{Brownian motions, persistent homology and chirality}, Journal of Applied and Computational Topology 9 (2025), article 24.

\bibitem{Be08}
C. Bene\v{s}, \emph{Counting planar random walk holes}, Annals of Probability 36 (2008), no.~1, 91--126.

\bibitem{CCL09}
F. Chazal, D. Cohen-Steiner, and A. Lieutier, \emph{A sampling theory for compact sets in Euclidean space}, Discrete \& Computational Geometry 41 (2009), no.~3, 461--479.

\bibitem{CDGO16}
F. Chazal, V. de Silva, M. Glisse, and S. Oudot, \emph{The Structure and Stability of Persistence Modules}, SpringerBriefs in Mathematics, Springer, Cham, 2016.

\bibitem{CEH07}
D. Cohen-Steiner, H. Edelsbrunner, and J. Harer, \emph{Stability of persistence diagrams}, Discrete \& Computational Geometry 37 (2007), no.~1, 103--120.

\bibitem{DV75}
M.D. Donsker and S.R.S. Varadhan, \emph{Asymptotics for the Wiener sausage}, Communications on Pure and Applied Mathematics 28 (1975), 525--565.

\bibitem{ELZ02}
H. Edelsbrunner, D. Letscher, and A. Zomorodian, \emph{Topological persistence and simplification}, Discrete \& Computational Geometry 28 (2002), no.~4, 511--533.

\bibitem{Fe59}
H. Federer, \emph{Curvature measures}, Transactions of the American Mathematical Society 93 (1959), 418--491.

\bibitem{Fu85}
J.H.G. Fu, \emph{Tubular neighborhoods in Euclidean spaces}, Duke Mathematical Journal 52 (1985), no.~4, 1025--1046.

\bibitem{Gr93}
K. Grove, \emph{Critical point theory for distance functions}, in: Differential Geometry: Riemannian Geometry (Los Angeles, CA, 1990), Proceedings of Symposia in Pure Mathematics 54, Part~3, American Mathematical Society, Providence, RI, 1993, pp.~357--385.

\bibitem{GuDrift}
T. Guillaume, \emph{Persistence of the Wiener sausage: sampling stability and limit theorems for Betti-curve functionals of drifted planar Brownian motion}, preprint, \url{https://arxiv.org/abs/2604.03130} [v3] (2026).

\bibitem{GuII}
T. Guillaume, \emph{The expected number of holes of the planar Wiener sausage}, preprint (2026). To appear soon on arxiv.org under this title.

\bibitem{GuPS}
T. Guillaume, \emph{Holes in planar parallel sets: an integrated Betti-number bound and its pointwise failure}, preprint, arXiv:2609.31691 (2026).

\bibitem{HST18}
Y. Hiraoka, T. Shirai, and K.D. Trinh, \emph{Limit theorems for persistence diagrams}, Annals of Applied Probability 28 (2018), no.~5, 2740--2780.

\bibitem{HNPS19}
N. Holden, \c{S}. Nacu, Y. Peres, and T.S. Salisbury, \emph{How round are the complementary components of planar Brownian motion?}, Annales de l'Institut Henri Poincar\'e, Probabilit\'es et Statistiques 55 (2019), no.~2, 882--908.

\bibitem{Ho14}
O. Honzl, \emph{On an upper bound of the Euler characteristic of the Wiener sausage}, Methodology and Computing in Applied Probability 16 (2014), no.~2, 331--353.

\bibitem{HLW04}
D. Hug, G. Last, and W. Weil, \emph{A local Steiner-type formula for general closed sets and applications}, Mathematische Zeitschrift 246 (2004), no.~1--2, 237--272.

\bibitem{La06}
G. Last, \emph{On mean curvature functions of Brownian paths}, Stochastic Processes and their Applications 116 (2006), no.~12, 1876--1891.

\bibitem{La96}
G.F. Lawler, \emph{The dimension of the frontier of planar Brownian motion}, Electronic Communications in Probability 1 (1996), 29--47.

\bibitem{LSW01}
G.F. Lawler, O. Schramm, and W. Werner, \emph{The dimension of the planar Brownian frontier is $4/3$}, Mathematical Research Letters 8 (2001), no.~4, 401--411.

\bibitem{LG90}
J.-F. Le Gall, \emph{Wiener sausage and self-intersection local times}, Journal of Functional Analysis 88 (1990), no.~2, 299--341.

\bibitem{LG91}
J.-F. Le Gall, \emph{On the connected components of the complement of a two-dimensional Brownian path}, in: Random Walks, Brownian Motion, and Interacting Particle Systems, Progress in Probability, vol.~28, Birkh\"auser, Boston, 1991, pp.~323--338.

\bibitem{LG92}
J.-F. Le Gall, \emph{Some properties of planar Brownian motion}, in: \'Ecole d'\'Et\'e de Probabilit\'es de Saint-Flour XX---1990, Lecture Notes in Mathematics 1527, Springer, Berlin, 1992, pp.~111--235.

\bibitem{Mo89}
T.S. Mountford, \emph{On the asymptotic number of small components created by planar Brownian motion}, Stochastics and Stochastics Reports 28 (1989), no.~3, 177--188.

\bibitem{Pe23}
D. Perez, \emph{On the persistent homology of almost surely $C^{0}$ stochastic processes}, Journal of Applied and Computational Topology 7 (2023), 611--650.

\bibitem{RSM09}
J. Rataj, E. Spodarev, and D. Meschenmoser, \emph{Approximations of the Wiener sausage and its curvature measures}, Annals of Applied Probability 19 (2009), no.~5, 1840--1859.

\bibitem{RSS09}
J. Rataj, V. Schmidt, and E. Spodarev, \emph{On the expected surface area of the Wiener sausage}, Mathematische Nachrichten 282 (2009), no.~4, 591--603.

\bibitem{Sp64}
F. Spitzer, \emph{Electrostatic capacity, heat flow and Brownian motion}, Zeitschrift f\"ur Wahrscheinlichkeitstheorie und Verwandte Gebiete 3 (1964), 110--121.

\bibitem{Th25}
A.M. Thomas, \emph{Convergence of persistence diagrams for discrete time stationary processes}, Journal of Applied and Computational Topology 9 (2025), 177--222.

\bibitem{We94}
W. Werner, \emph{Sur la forme des composantes connexes du compl\'ementaire de la courbe brownienne plane}, Probability Theory and Related Fields 98 (1994), no.~3, 307--337.

\end{thebibliography}
\end{document}